\documentclass[12pt, reqno] {amsart}

\date{\today}          
\usepackage{hyperref}
\usepackage{latexsym,amsmath,amsfonts,amscd,amssymb}
\usepackage{lscape}
\usepackage{array}
\input xy
 \xyoption{all}

\DeclareFontFamily{OT1}{rsfs}{}
\DeclareFontShape{OT1}{rsfs}{n}{it}{<->rsfs10}{}
\DeclareMathAlphabet{\curly}{OT1}{rsfs}{n}{it}

\theoremstyle{plain}  
\newtheorem{theorem}{Theorem}[section]

\newtheorem*{theorem*}{Theorem}

\newtheorem{corollary}[theorem]{Corollary}
\newtheorem{lemma}[theorem]{Lemma}
\newtheorem{proposition}[theorem]{Proposition}

\theoremstyle{definition}
\newtheorem{definition}[theorem]{Definition}

\theoremstyle{remark}
\newtheorem{example}[theorem]{Example}

\newtheorem*{notation*}{Notation}
\newtheorem{remark}[theorem]{Remark}

\newtheorem*{claim*}{Claim}
\numberwithin{equation}{section}

\renewcommand{\leq}{\leqslant}
\renewcommand{\le}{\leqslant}
\renewcommand{\geq}{\geqslant}
\renewcommand{\ge}{\geqslant}
\renewcommand{\setminus}{\smallsetminus}

\newcommand{\R}{\mathbb{R}}

\newcommand{\Z}{\mathbb{Z}}
\newcommand{\C}{\mathbb{C}}

\newcommand{\cM}{\mathcal{M}}

\newcommand{\xra}{\xrightarrow}

\newcommand{\PSL}{\mathrm{PSL}}

\newcommand{\U}{\mathrm{U}}
\newcommand{\OO}{\mathrm{O}}
\newcommand{\GL}{\mathrm{GL}}
\newcommand{\SL}{\mathrm{SL}}

\newcommand{\SO}{\mathrm{SO}}
\newcommand{\Sp}{\mathrm{Sp}}

\DeclareMathOperator{\Jac}{Jac}
\DeclareMathOperator{\ad}{ad}

\DeclareMathOperator{\grad}{grad}

\DeclareMathOperator{\rank}{rank}

\DeclareMathOperator{\Hom}{Hom}
\DeclareMathOperator{\End}{End}

\DeclareMathOperator{\Sym}{Sym}

\newcommand{\noi}{\noindent}

\newcommand{\aut}{\operatorname{aut}}
\newcommand{\Aut}{\operatorname{Aut}}

\newcommand{\liem}{\mathfrak{m}}

\newcommand{\lieh}{\mathfrak{h}}
\newcommand{\liehc}{\mathfrak{h}^{\mathbb{C}}}
\newcommand{\lieg}{\mathfrak{g}}
\newcommand{\liegc}{\mathfrak{g}^{\mathbb{C}}}

\newcommand{\CC}{\mathbb{C}}

\newcommand{\lie}{\mathfrak}

\newcommand{\hlie}{\mathfrak{h}}

\newcommand{\mlie}{\mathfrak{m}}

\newcommand{\mclie}{\mlie^{\CC}}

\newcommand{\HC}{H^{\CC}}

\DeclareMathOperator{\YM}{YM}
\DeclareMathOperator{\YMH}{YMH}
\DeclareMathOperator{\dvol}{dvol}
\DeclareMathOperator{\Lie}{Lie}
\DeclareMathOperator{\Ric}{Ric}
\DeclareMathOperator{\Rep}{Rep}

\begin{document}

\title{Morse theory, Higgs fields, and Yang-Mills-Higgs functionals}

\author[S. B. Bradlow]{Steven B. Bradlow}
\address{Department of Mathematics \\
University of Illinois \\
Urbana \\
IL 61801 \\
USA }
\email{bradlow@math.uiuc.edu}

\author[G. Wilkin]{Graeme Wilkin}
\address{Department of Mathematics\\
National University of Singapore \\
Block S17 \\
10, Lower Kent Ridge Rd \\
Singapore 119076}
\email{graeme@nus.edu.sg}

\dedicatory{Dedicated to Professor Dick Palais on the occasion
     of his $80$th birthday}

\thanks{ }


\maketitle

\section{Introduction}

In classical Morse theory, manifolds are smooth and finite dimensional, and critical points are isolated and non-degenerate.  Nearly 50 years ago Palais expanded the theory to cover infinite dimensional manifolds. Subsequent developments showed how to cope if  critical points are replaced by non-degenerate critical submanifolds (Bott, \cite{Bott54}) and if the non-degeneracy condition is relaxed to allow (possibly singular) critical sets that are contained in a minimising manifold (Kirwan, \cite{Kirwan84}). The full range of such phenomena appear in the settings described in this paper, with key ideas descending directly from Palais's work.  Given that the first author is a mathematical grandchild and the second author is from the next generation in the Palais family tree, this birthday offering is thus both by,  and about, Palais's mathematical descendants. 

The setting for this paper is the theory of Higgs bundles over a closed Riemann surface, say $X$.  Higgs bundles were first introduced by Nigel Hitchin in his landmark 1987 paper \cite{Hitchin87}, though the name was first introduced by the other pioneer in the subject, namely Carlos Simpson, in his 1988 paper \cite{Simpson88}.  In both cases, the objects introduced correspond to Higgs bundles associated to complex groups (actually $\SL(n,\C)$); in this paper we will adopt the more general notion (introduced by Hitchin in \cite{Hitchin92}) of Higgs bundles for real reductive Lie groups. We call these objects $G$-Higgs bundles, where $G$ is the real group. This will allow us to explore a wider range of phenomena.  In the spirit of this review, we focus on phenomena detected using Morse-theoretic ideas. 

The main objects of study are the finite dimensional moduli spaces of $G$-Higgs bundles.  A fundamental feature of these spaces is that they admit several complementary interpretations. They have an algebraic interpretation as a moduli space of polystable objects in the sense of Geometric Invariant Theory, but can also be viewed as moduli spaces of solutions to gauge-theoretic equations.  In some cases they can be identified with a representation variety for representations of the fundamental group $\pi_1(X)$. Related to these different points of view, there are two natural ways to apply Morse-theoretic  ideas to the study of the topology of these moduli spaces. 

The first way is to view the moduli space as a quotient space of the infinite dimensional space of connections and Higgs fields on a principal bundle. The Morse function in this setting is a natural energy functional, the so-called Yang-Mills-Higgs functional, which, crucially, satisfies an equivariant version of the famous Palais-Smale Condition C along complex gauge orbits. The absolute minimum of this functional is a gauge invariant analytic subspace and the orbit space is the moduli space of interest.  This finite dimensional moduli space has its own natural choice of Morse function known as the Hitchin function. The second type of Morse theory is based on this Hitchin function.

After introducing $G$-Higgs bundles in Section \ref{GHiggs},  in Section \ref{sect:YMH} we describe the infinite dimensional Morse theory for the Yang-Mills-Higgs functional and in Section \ref{sect:hitchinfn}  we discuss the finite dimensional Morse Theory for the Hitchin function. This paper is mostly a review of a still-evolving body of work and and, apart from Theorem \ref{newtheorem} in Section \ref{sect:gradflow}, contains essentially no new results.

\section{$G$-Higgs bundles}\label{GHiggs}

\subsection{Definitions}\label{subs:defns}
Let $G$ be a \textbf{real reductive Lie group}. While the precise definition requires some care (see for example Knapp \cite[p.~384]{knapp:1996}), it is sufficient for our purposes to take $G$ to be the real form of a complex reductive group. In fact, in this paper the main complex groups that will appear are $\GL(n,\C)$ and $\SL(n,\C)$.  The definition of a $G$-Higgs bundle requires the following additional data:
\begin{itemize}
\item a maximal compact subgroup $H\subset G$
\item a Cartan involution $\theta$ giving a decomposition of the Lie algebra 
\begin{displaymath}
\lie{g} = \lieh \oplus \liem
\end{displaymath}
into its $\pm1$-eigenspaces, where  $\lieh$ is the Lie
algebra of $H$, and
\item a non-degenerate Ad-invariant and $\theta$-invariant bilinear form, $B$, on $\lieg$.
\end{itemize}

\noi Given the above choices, the Lie algebra satisfies
$$[\hlie,\hlie]\subset\hlie,\qquad
[\hlie,\mlie]\subset\mlie,\qquad [\mlie,\mlie]\subset\hlie.$$ 

\noi and hence the adjoint representation restricts to define a linear representation of $H$ on $\liem$.  This action extends to a linear holomorphic action 

\begin{equation} \label{eq:isotropy-rep}
\iota\colon \HC \to \GL(\mclie).    
\end{equation}

\noi called the \textbf{isotropy representation}.

\begin{definition}\label{def:g-higgs}
A \textbf{$G$-Higgs bundle} on $X$ is a pair
$(E_{\HC},\varphi)$, where $E_{\HC}$ is a holomorphic $\HC$-principal bundle
over $X$ and $\varphi$ is a holomorphic section of $E_{\HC}(\mclie)\otimes
K$,  where $E_{\HC}(\mclie)= E_{\HC} \times_{\HC}\mclie$ is the
$\mclie$-bundle associated to $E_{\HC}$ via the isotropy representation and $K$ is the canonical bundle of $X$.
\end{definition}

\begin{example}
If $G$ is compact them $H=G$ and $\liem=\{0\}$. A $G$-Higgs bundle is thus equivalent to a holomorphic $G^{\C}$-bundle.
\end{example}

\begin{example}\label{ex:Gc} If $G$ is the underlying real group of a complex reductive group then $\HC=G$ and $\mclie=i\lieg=\lieg$.  In this case a $G$-Higgs bundle is a pair $(E,\varphi)$, where $E$ is a holomorphic $G$-principal bundle over $X$ and $\varphi$ is a holomorphic section of $ad(E)\otimes K$. In particular, if $G=\GL(n,\C)$ and we replace the principal $G$-bundle by the rank $n$ vector bundle determined by the standard representation, then a {\bf $\GL(n,\C)$-Higgs bundle} is a pair $(V,\varphi)$, where $V$ is a rank $n$ holomorphic vector bundle over $X$ and $\varphi$ is a holomorphic map $\varphi:V\rightarrow V\otimes K$. If we take $G=\SL(n,\C)$ then an {\bf $\SL(n,\C)$-Higgs bundle} is also a pair $(V,\varphi)$ but $\det(V)$ is trivial and $\varphi$ has zero trace.
\end{example}

\begin{example} In Table 1  we describe the structure of $\SL(n, \R)$, $\Sp(2n, \R)$ and $\U(p, q)$ as real reductive groups. In each case there is a Cartan involution given by $\theta(u) = - u^*$, which induces the Cartan decomposition given in the table.  Table 2 describes the associated $G$-Higgs bundles and their vector bundle representation for these examples of real reductive groups.  Thus for example the data defining a $\U(p,q)$-Higgs bundle on a Riemann surface $X$ is equivalent to the tuple $(V_1,V_2,\beta,\gamma)$ where $V_1$ and $V_2$  are respectively rank $p$ and rank $q$ holomorphic vector bundles on $X$ and the maps

\begin{align*}\beta: &V_1\longrightarrow V_2\otimes K\\ 
\gamma: & V_2\longrightarrow V_1\otimes K
\end{align*}

\noi define the Higgs field (here $K$ is the canonical bundle on $X$). \footnotemark\footnotetext{For more details see \cite{bradlow-garciaprada-gothen2003}, \cite {garcia-gothen-mundet:2008}, \cite{Hitchin92}.} The case of $\Sp(2n,\R)$ is explained further in Example \ref{Ex;Sp2nR}.

\begin{table}\label{Cartan}
\begin{tabular}{|c|c|c|c|}
\hline
$G$ & $H$ & $H^{\C}$ & $\liem^{\C}$\\
\hline
$\SL(n,\R)$ & $\SO(n)$ & $\SO(n,\C)$ & $\Hom^{sym}_0(\C^n,\C^n)$\\
\hline
$\Sp(2n,\R)$ & $\U(n)$ & $\GL(n,\C)$ & $ \Sym^2(\C^n)\oplus\Sym^2((\C^n)^*)$\\
\hline
$\U(p,q)$ & $\U(p)\times\U(q)$ &{\tiny $\GL(p,\C)\times\GL(q,\C)$} &  $ \Hom(\C^p,\C^q)\oplus 
\Hom(\C^q,\C^p)$\\
\hline
\end{tabular}
\caption{\small{Examples of Cartan decompositions of real reductive groups}}
\end{table}


\begin{table}
\begin{tabular}{|c|c|c|}
\hline
$G$ & Bundle(s) & Higgs field\\
\hline
$\SL(n,\R)$ &$V$ (rank n) &$\varphi:V\longrightarrow V\otimes K$\\
&orthogonal & symmetric, $Tr(\varphi)=0$\\
\hline
$\Sp(2n,\R)$ & $V$ (rank n)& $\varphi=(\beta, \gamma )$,  
$\begin{cases}\beta: V^*\longrightarrow V\otimes K\\ \gamma: V\longrightarrow V^*\otimes K\end{cases}$ \\
& (degree d)  &$\beta^t=\beta; \gamma^t=\gamma$ \\
\hline
$\U(p,q)$ &$V_1$ (rank $p$)  &$\varphi=(\beta, \gamma )$,  
$\begin{cases}\beta: V_1\longrightarrow V_2 \otimes K\\ \gamma: V_2\longrightarrow V_1\otimes K \end{cases}$ \\  & $V_2$ (rank $q$)  & \\
\hline
\end{tabular}
\caption{\small{The G-Higgs bundles with $E_{H^{\C}}$ replaced by vector bundle(s)}}
\end{table}

\end{example}

\subsection{Moduli spaces}\label{subs:modspaces}

Two $G$-Higgs pairs $(E_{\HC},\varphi)$ and
$(E_{\HC}',\varphi')$ are \textbf{isomorphic} if there is an isomorphism
$f\colon E_{\HC} \xra{\simeq} E_{\HC}'$ such that $\varphi = f^*\varphi'$ where $f^*$ is the obvious induced map.  Loosely speaking, a moduli space for $G$-Higgs bundles is a geometric space whose points parameterize isomorphism classes of pairs $(E,\varphi)$.  In general the space of all isomorphism classes has bad topological properties; for example, it may not even be Hausdorff.  One standard solution to this problem is provided by Geometric Invariant Theory (GIT) notions of stability and polystability.  The point of these stability properties is to identify the isomorphism classes for which a suitable moduli space can be constructed.   Constructions of this sort for moduli spaces of holomorphic bundles have a long history going back to the work of Narasimhan and Seshadri (\cite{NarasimhanSeshadri65}).  The appropriate notions of stability and polystability have been formulated in very general terms  for $G$-Higgs bundles in \cite{garcia-prada-gothen-mundet:2009a}. While the general definitions are quite cumbersome, in many special cases they simplify to conditions akin to slope stability/semistability/polystability for holomorphic bundles. 

\begin{example}\label{Ex:GLnC} Let $G=\GL(n,\C)$ and let  $(V,\Phi)$  be a $G$-Higgs bundle over a Riemann surface $X$ (as in Example \ref{ex:Gc}).  The Higgs bundle $(V,\Phi)$ is defined to be semistable if 

\begin{equation}\label{GLnstable}
\frac{\deg(V')}{\rank(V')}\le\frac{\deg(V)}{\rank(V)}
\end{equation}

\noi for all $\Phi$-invariant subbundles $V'\subset V$.  The Higgs bundle is stable if the inequality in \eqref{GLnstable} is strict for all proper invariant subbundles, and it is polystable if it decomposes as a direct sum of stable Higgs bundles all with the same slope (where the slope of a bundle $V$ is the ratio $\frac{\deg(V)}{\rank(V)}$).
\end{example}

\begin{definition}\label{def:ModG} The \textbf{moduli space of polystable
$G$-Higgs bundles} $\mathcal{M}(G)$ is the set of isomorphism
classes of polystable $G$-Higgs bundles $(E_{\HC},\varphi)$.
\end{definition}

The topological type of the bundle $E_{\HC}$ is an invariant of connected components of $\mathcal{M}(G)$ so we refine our definition to identify components by the possible topological types. The topological types can be identified by suitable characteristic classes such as Chern classes (if $\HC=\GL(n,\C)$) or Stiefel-Whitney classes (for $\OO(n,\C)$-bundles). If $G$ is semisimple and connected the topological type of a principal $H^\CC$-bundle on $X$ is classified by an element in  $ \pi_1(H^\CC)=\pi_1(H)=\pi_1(G)$.  We will loosely denote the characteristic class by $c(E_{\HC})$ and denote its value by $c(E_{\HC})=d$.  We can thus define; 

\begin{definition}\label{def:MdG} The component $\mathcal{M}_d(G)\subset \mathcal{M}(G)$, called the  \textbf{moduli space of polystable $G$-Higgs bundles of type $d$}  is the set of isomorphism classes of polystable $G$-Higgs bundles $(E_{\HC},\varphi)$ such that $c(E_{\HC})=d$.  
\end{definition}

\begin{remark} When $G$ is compact, the moduli space $\cM_d(G)$ coincides
with $M_d(G^\CC)$, the moduli space of polystable principal $G^\CC$-bundles
with topological invariant $d$.
\end{remark}

The moduli space $\cM_d(G)$ has the structure of a complex analytic
variety.  This can be seen by the standard slice method (see, e.g.,
Kobayashi \cite[Ch. VII]{kobayashi:1987}).  Geometric Invariant Theory
constructions are available in the literature for $G$ real compact
algebraic (Ramanathan \cite{ramanathan:1975, ramanathan:1996}) and for $G$ complex
reductive algebraic (Simpson \cite{simpson:1994,simpson:1995}).  The
case of a real form of a complex reductive algebraic Lie group follows
from the general constructions of Schmitt
\cite{schmitt:2008}. We thus have the following.

\begin{theorem}\label{alg-moduli}
The moduli space $\cM_d(G)$ is a complex analytic variety, which is
algebraic when $G$ is algebraic.
\end{theorem}

The moduli spaces are not in general smooth but the singularities can be identified.  The singular points are associated with the $G$-Higgs bundles which are polystable but not stable.  This allows the moduli spaces to be smooth in special cases where polystability implies stability.  For example in the case of $G=\GL(n,\C)$ (as in Example \ref{Ex:GLnC})   if the rank and degree of the underlying vector bundle are coprime then polystability is equivalent to stability and hence the moduli spaces are smooth.

The moduli spaces can also be viewed as quotients in an infinite dimensional gauge-theoretic setting. Fixing a smooth principal $\HC$-bundle $E$ with topological invariant $c(E)=d$, the holomorphic principal $\HC$-bundles with the same topological invariant can be described by a choice of holomorphic structure on $E$ .  If $E$ is a vector bundle then holomorphic structures are defined by $\C$-linear operators $\bar{\partial}_E : \Omega^0(E) \rightarrow \Omega^{0,1}(E)$ that satisfy the Leibniz rule \footnotemark\footnotetext{ In general there is also an integrability condition but on Riemann surfaces this is automatically satisfied (see \cite{NewlanderNirenberg57} for the original theorem on complex manifolds, and \cite{AtiyahHitchinSinger78} for the bundle case).}. This construction can be adapted for principal bundles, and in all cases the space of holomorphic structures, denoted by $\mathcal{A}^{0,1}(E_{\HC})$, is an infinite dimensional affine space locally modeled on the space of anti-holomoprhic 1-forms with values in the endomorphism bundle, $\Omega^{0,1}(E_{\HC}(\hlie^{\C}))$.  Moreover a given holomorphic structure induces an operator 
$\bar{\partial}_E$ on $\Omega^{1,0}(E_{\HC}(\mclie))$ such that the condition for a 1-form with values in $E_{\HC}(\mclie)$ to be holomorphic is $\bar{\partial}_E\varphi=0$.  

\begin{definition} The {\bf configuration space of $G$-Higgs bundles on $E$}, i.e.  the space of all $G$-Higgs bundles for which the principal $\HC$-bundle has topological invariant $d$, is the space

\begin{equation}
\mathcal{B}_d(G)=\{(\bar{\partial}_E,\varphi)\in \mathcal{A}^{0,1}(E_{\HC})\times\Omega^{1,0}(E_{\HC}(\mclie))\ |\  \bar{\partial}_E\varphi=0\ \}
\end{equation}
\end{definition}

The complex gauge group for $E_{\HC}$, i.e. the group of sections $\mathcal{G}^{\C}=\Omega^0(Ad(E_{\HC)}))$, acts on $\mathcal{B}_d(G)$ in such a way that $\mathcal{G}^{\C}$-orbits correspond to isomorphism classes of $G$-Higgs bundles.  Thus

\begin{equation}\label{complex}
\mathcal{M}_d(G)=\mathcal{B}_d^{ps}(G)/\mathcal{G}^{\C}
\end{equation}

\noi where $\mathcal{B}_d^{ps}(G)\subset \mathcal{B}_d(G)$ denotes the set of polystable $G$-Higgs bundles. 

The Hitchin-Kobayashi correspondence for Higgs bundles identifies polystability with the existence of solutions to gauge theoretic equations, and thus yelds yet another interpretation of the moduli spaces. There are two ways to formulate the equations.  In the first, one needs to fix a reduction of structure group on $E_{\HC}$ to the compact group $H$.  If $\HC=\GL(n,\C)$ so $H=U(n)$ then this corresponds to fixing a hermitian bundle metric on the associated rank $n$ vector bundle. We thus refer to this in all cases as a choice of metric. The metric allows us to write $E_{\HC}=E_H\times_H\HC$, where $E_H$ is a principal $H$-bundle, and defines the real gauge group $\mathcal{G}\subset\mathcal{G}^{\C}$ consisting of the sections of $\Omega^0(Ad(E_{H)}))$.   It provides, through the construction of Chern connections, an identification between holomorphic structures on $E_{\HC}$ and connections on $E_H$.  Denoting the space of connections on $E_H$ by $\mathcal{A}(E_H)$, we can thus formulate the following set of equations for pairs $(A,\varphi)\in  \mathcal{A}(E_{H})\times\Omega^{1,0}(E_{\HC}(\mclie))$:

\begin{align}
&F_A-[\varphi,\tau(\varphi)]=c\omega\label{Higgsequations1}\\
&d_A^{0,1}\varphi=0\label{Higgsequations2}
\end{align}

\noi Here $F_A$ is the curvature of the connection, $\tau$ denotes the compact conjugation on $\liegc$, i.e. the anti-linear involution which defines the compact real form, $c$ is an element in the center of $\lieh$, $\omega$ is the Kahler form on $X$, and  $d_A$ is the antiholomorphic part of the covariant derivative induced by $A$ on $E_{\HC}(\mclie))$. The constant $c$ is determined, via Chern-Weil theory,  by the topological invariant $d$. The operator $d_A^{0,1}$ defines the holomorphic structure on $E_{\HC}$, so the second equation says that the Higgs field $\varphi$ is holomorphic. 

Going back to the identification 

\begin{equation}
 \mathcal{A}(E_{H})\simeq\mathcal{A}^{0,1}(E_{\HC})
\end{equation}

\noi determined by a choice of metric,  we can regard the holomorphic structure and (holomorphic) Higgs field as given and view the equation \eqref{Higgsequations1} as an equation for a metric on $E_{\HC}$.  The Hitchin-Kobayashi correspondence then asserts:

\noi {\bf Hitchin-Kobayashi Correspondence\footnotemark\footnotetext{The original correspondence of this sort goes back to a theorem of Narasimhan and Seshadri (\cite{NarasimhanSeshadri65})for vector bundles over closed Riemann surfaces, while the versions for bundles over closed Kahler manifolds go back to \cite{Lubke} ,\cite{Donaldson87}, \cite{UhlenbeckYau86}. The scope of the result has by now been greatly extended to cover bundles with various kinds of additional data and also more general base manifolds (see \cite{LT} for a summary).  For the versions most appropriate for $G$-Higgs bundles see \cite{bradlow-garcia-prada-mundet:2003}, and\cite{LT}.}
{\it The $G$-Higgs bundle $(E_{\HC},\varphi)$ is polystable if and only if $E_{\HC}$ admits a metric satisfying the equation \eqref{Higgsequations1}}}

\noi On the other hand, the equation \eqref{Higgsequations1} can be interpreted as a symplectic moment map condition for the action of the gauge group $\mathcal{G}$ on the infinite dimensional space $\mathcal{A}(E_{H})\times\Omega^{1,0}(E_{\HC}(\mclie))$.  The  symplectic form on this space is defined by the $L^2$-inner products  constructed using the metric on $E_H$ together with the metric on $X$.  The resulting metric on $\mathcal{B}(d)$ is K\"ahler and hence defines a symplectic structure on the smooth locus of the moduli space. The action of the real gauge group is hamiltonian with respect to this symplectic structure and has a moment map 

\begin{align}
\Psi: \mathcal{A}(E_{H})&\times\Omega^{1,0}(E_{\HC}(\mclie))\rightarrow \Lie(\mathcal{G})^* \cong \Omega^2(\ad(E_H)) \\
(A,\varphi)&\mapsto \Lambda(F_A-[\varphi,\tau(\varphi)])- c
\end{align}

\noi where $\Lambda$ denotes contraction against the K\"ahler form $\omega$. If we define a subspace of $\mathcal{A}(E_{H})\times\Omega^{1,0}(E_{\HC}(\mclie))$ by

\begin{equation}
\mathcal{B}_d^H(G)=\{(A,\varphi)\in \mathcal{A}(E_{H})\times\Omega^{1,0}(E_{\HC}(\mclie))\ |\  d_A^{0,1}\varphi=0\ \}
\end{equation}

\noi and restrict the moment map to  $\mathcal{B}_d^H(G)$ we can thus describe the moduli space as a symplectic quotient, i.e

\begin{equation}\label{sympquot}
\mathcal{M}_d(G)=\Psi^{-1}(0)/\mathcal{G} 
\end{equation}

\noi It is useful to make use of all three descriptions of the moduli spaces $\mathcal{M}_d(G)$, namely 

\begin{itemize}
\item as a space of isomorphism classes of polystable objects (Definition \ref{def:MdG}),
\item as a complex quotient (see \eqref{complex}), and
\item as a symplectic quotient (see \eqref{sympquot})
\end{itemize}

There is one further interpretation of the moduli spaces which relies on the fact that while the Higgs field $\varphi$ is a 1-form which takes its values in $\mclie$, the combination $\theta=\varphi-\tau(\varphi)$ takes its values in $\liem$.  It follows that 

\begin{equation}
D=d_A+\theta
\end{equation}

\noi takes its values in $\lieg=\lieh\oplus\liem$ and defines a connection on the principal $G$-bundle 

$$E_G=E_H\times_HG\ ,$$

\noi  i.e. on the bundle obtained from $E_H$ by extending the structure group.  The bundle $E_G$ is thus, by construction,  a $G$ bundle with a reduction of structure group to $H$, i.e. with a metric.  We can reformulate equations \eqref{Higgsequations1} and \eqref{Higgsequations2} in terms of $D$ and the metric on $E_G$.  If we assume that $G$ is semisimple and hence that the right hand side of \eqref{Higgsequations1} is zero, then the equations become

\begin{align}\label{eqn:realHiggs}
F_D & =0\nonumber \\
d_A^*\theta & =0 \ .
\end{align}

\noi The first equation says that the connection $D$ is flat and hence that its holonomy defines a representation

\begin{equation}\label{eqn:rho}
\rho:\pi_1(X)\rightarrow G \ .
\end{equation}

\noi Using the flat structure defined by $D$, the metric on $E_G$ is equivalent to a $\pi_1(X)$-equivariant map from the universal cover of $X$:

\begin{equation}
\sigma: \tilde{X}\rightarrow G/H \ .
\end{equation}

\noi The second equation says that this map is {\it harmonic} with respect to the metric induced on $\tilde{X}$ from that on $X$ and the invariant metric on $G/H$.  An important theorem of Corlette asserts

\begin{theorem}\cite{corlette:1988} The representation defined by \eqref{eqn:rho} is reductive and all reductive representations of $\pi_1(X)$ in $G$ arise in this way, i.e. the corresponding flat $G$-bundles admit harmonic metrics.
\end{theorem}

If $X$ is a closed Riemann surface of genus $g$ then $\Hom(\pi_1(X),G)$, i.e. the set of all surface group representations in $G$, forms a real analytic subspace of the $2g$-fold product $G\times\dots\times G$. The group $G$ acts by conjugation on this space. The quotient $\Hom(\pi_1(X),G)/G$ does not in general have good geometric or even topological properties but restricting to the reductive representations yields a good {\bf moduli space of representations}

\begin{equation}\label{RepG}
\Rep(\pi_1(X),G)=\Hom^+(\pi_1(X),G)/G
\end{equation}

\noi (where the $+$ denotes the restriction to the reductive representations).  The topological invariant $d\in\pi_1(G)$ which labels components of $\mathcal{M}_d(G)$ also labels components of $\Rep(\pi_1(X),G)$, leading to the obvious definition for spaces $\Rep_d(\pi_1(X),G)$.  For semisimple groups $G$, Corlette's theorem together with the equivalence between equations \eqref{eqn:realHiggs} and the Higgs bundle equations \eqref{Higgsequations1} and \eqref{Higgsequations2} thus allows us to identify

\begin{equation}\label{M=R}
\mathcal{M}_d(G)=\Rep_d(\pi_1(X),G)\ .
\end{equation}

\noi This identification, especially when coupled with the Morse theoretic methods described in  the rest of this paper, turns Higgs bundles into an effective tool for studying the moduli spaces of surface group representations.

\subsection{Morse functions}

There are two functions which turn out to be useful for exploring the topology of the moduli spaces $\mathcal{M}_d(G)$, i.e. which have good Morse-theoretic properties. The first, defined on the infinite dimensional configuration space and denoted by

\begin{equation}\label{eqn:YMH}
\YMH : \mathcal{A}(E_{H})\times\Omega^{1,0}(E_{\HC}(\mclie)) \rightarrow \R\ ,
\end{equation}

\noi is a generalization of the Yang-Mills functional on the space of connections.  The second, defined on $\mathcal{M}_d(G)$ itself and denoted by

\begin{equation}\label{Hitchinf}
f : \mathcal{M}_d(G) \rightarrow \R\ ,
\end{equation}

\noi was introduced by Hitchin and relies on the presence of the Higgs field. Both are related to symplectic moment maps.  In the next sections we explore these two functions from the point of view of Morse theory.

\section{The Yang-Mills-Higgs functional}\label{sect:YMH}

The goal of this section is to describe recent developments on the infinite-dimensional Morse theory of the Yang-Mills-Higgs functional. Firstly, we recall the work of Atiyah \& Bott on the Morse theory of the Yang-Mills functional for bundles over a compact Riemann surface. 

\subsection{Background on the Yang-Mills functional and semistable holomorphic bundles}

In the late 1970s and early 1980s, Atiyah and Bott pioneered the use of the Yang-Mills functional as an equivariant Morse function to study the topology of the moduli space of semistable holomorphic bundles over a compact Riemann surface (see \cite{AtiyahBott81} and \cite{AtiyahBott83}). 

\subsubsection{The algebraic picture}

The setup is as follows. Fix a smooth complex vector bundle $E$ over a compact Riemann surface $X$.  This corresponds to taking $G=\U(n)$ in the terminology of Section \ref{subs:modspaces}.  As described in Section \ref{subs:modspaces}, the space $\mathcal{A}^{0,1}(E)$ of holomorphic structures on $E$ can be identified with an affine space locally modeled on $\Omega^{0,1}(\End(E))$, i.e.
\begin{equation*}
\mathcal{A}^{0,1}(E) \cong \bar{\partial}_{A_0} + \Omega^{0,1}(\End(E)) 
\end{equation*}
for a fixed operator $\bar{\partial}_{A_0}$.


Any holomorphic bundle has a canonical \emph{Harder-Narasimhan filtration}
\begin{equation*}
X \times \{ 0 \} = E_0 \subset E_1 \subset \cdots \subset E_k = E 
\end{equation*}
by holomorphic sub-bundles, where, for each $\ell=1, \ldots, k$, the quotient bundle $E_\ell / E_{\ell-1}$ is the maximal semistable sub-bundle of $E / E_{\ell-1}$. The \emph{type} of the Harder-Narasimhan filtration is the vector $( \nu_1, \ldots, \nu_n)$, consisting of the slopes of the quotient bundles $E_\ell / E_{\ell-1}$ in the Harder-Narasimhan filtration, counted with multiplicity $\rank(E_\ell / E_{\ell-1})$. Let $\mathcal{A}_\lambda^{0,1}$ denote the set of holomorphic structures with Harder-Narasimhan type $\lambda$, and let $\mathcal{A}_{ss}^{0,1}$ denote the semistable holomorphic structures. There is a partial ordering on the Harder-Narasimhan type (see \cite{Shatz77}), and Atiyah and Bott show that the stratification 
\begin{equation*}
\mathcal{A}^{0,1} = \mathcal{A}_{ss}^{0,1} \cup \bigcup_{\lambda} \mathcal{A}^{0,1}_\lambda
\end{equation*}
by type has the following properties.
\begin{enumerate}
\item $\displaystyle{\overline{\mathcal{A}^{0,1}_\lambda} \subseteq \bigcup_{\mu \geq \lambda} \mathcal{A}^{0,1}_\mu}$.

\item $\displaystyle{H_\mathcal{G}^* \left( \mathcal{A}^{0,1}_\mu \right) \cong \bigoplus_{\ell=1}^k H_{\mathcal{G}_\ell}^* \left( \mathcal{A}^{0,1}_{ss}(E_\ell / E_{\ell-1}) \right) }$, where $E_1, \ldots, E_\ell$ are the bundles in the Harder-Narasimhan filtration, and $\mathcal{G}_\ell$ is the gauge group of $E_\ell / E_{\ell-1}$ for $\ell=1, \ldots, k$.

\item Each $\mathcal{A}^{0,1}_\lambda$ is a manifold with constant codimension $n_\lambda$, and there is an isomorphism
\begin{equation*}
H_\mathcal{G}^*\left( \bigcup_{\mu \leq \lambda} \mathcal{A}^{0,1}_\mu, \bigcup_{\mu < \lambda} \mathcal{A}^{0,1}_\mu \right) \cong H_\mathcal{G}^{*-n_\lambda} \left( \mathcal{A}^{0,1}_\mu \right) .
\end{equation*}

\item The stratification is \emph{equivariantly perfect} in that, for each type $\lambda$, the long exact sequence in $\mathcal{G}$-equivariant cohomology
\begin{equation*}
\cdots \rightarrow H_\mathcal{G}^*\left( \bigcup_{\mu \leq \lambda} \mathcal{A}^{0,1}_\mu, \bigcup_{\mu < \lambda} \mathcal{A}^{0,1}_\mu \right) \rightarrow H_\mathcal{G}^*\left( \bigcup_{\mu \leq \lambda} \mathcal{A}^{0,1}_\mu \right) \rightarrow H_\mathcal{G}^* \left( \bigcup_{\mu < \lambda} \mathcal{A}^{0,1}_\mu \right) \rightarrow \cdots 
\end{equation*}
splits into short exact sequences.
\end{enumerate} 

As a consequence of the above properties, one can (a) inductively compute the equivariant Poincar\'e polynomial of $\mathcal{A}^{0,1}_{ss}$ in terms of the equivariant Poincar\'e polynomials of all the lower rank spaces of semistable bundles, and (b) show that the inclusion $\mathcal{A}^{0,1}_{ss} \hookrightarrow \mathcal{A}^{0,1}$ induces a surjective map
\begin{equation*}
\kappa : H_\mathcal{G}^*(\mathcal{A}^{0,1}) \rightarrow H_\mathcal{G}^*(\mathcal{A}^{0,1}_{ss}) .
\end{equation*}

This map is known as the \emph{Kirwan map}. In \cite{Kirwan84}, Kirwan shows that the strategy described above applies in much more generality: it extends to compact symplectic manifolds with a Hamiltonian action of a compact connected Lie group, and that one can (a) inductively compute the Poincar\'e polynomial of the symplectic quotient, and (b) show that the Kirwan map is surjective.

\subsubsection{Morse theory of the Yang-Mills functional}

Now suppose that $E$ also has a Hermitian structure.  In the notation of Section \ref{subs:modspaces} this means a reduction of structure group from $\HC=\GL(n,\C)$ to $H=\U(n)$. For simplicity of notation, in this section we omit the subscripts $\HC$ or $H$. Then there is an identification of $\mathcal{A}^{0,1}(E)$ with the space $\mathcal{A}(E)$ of connections compatible with this structure (note that in higher dimensions, the integrability condition requires that we restrict attention to the subset of connections $d_A \in \mathcal{A}(E)$ whose curvature $F_A$  satisfies $F_A \in \Omega^{1,1}(\ad E)$). Therefore, there is an induced stratification 
\begin{equation}\label{eqn:connections-HN-stratification}
\mathcal{A} = \bigcup_{\lambda} \mathcal{A}_\lambda
\end{equation}
by the Harder-Narasimhan type of the corresponding holomorphic structure. Even though a change in the Hermitian metric will change the identification $\mathcal{A}^{0,1} \cong \mathcal{A}$, for a fixed connection $d_A \in \mathcal{A}$ the Harder-Narasimhan type of the associated holomorphic structure does not change when the metric changes. Therefore the stratification \eqref{eqn:connections-HN-stratification} is well-defined independently of the choice of metric.

The \emph{Yang-Mills functional}, $\YM : \mathcal{A} \rightarrow \R$, is defined by
\begin{equation*}
\YM(d_A) = \| F_A \|^2 = \int_X \left| F_A \right|^2 \dvol ,
\end{equation*}
and so we also have an induced functional $\YM : \mathcal{A}^{0,1} \rightarrow \R$, which depends on the identification $\mathcal{A}^{0,1} \cong \mathcal{A}$, and therefore on the choice of Hermitian metric.

The theorem of Narasimhan and Seshadri, i.e. the original Hitchin-Kobayashi correspondence mentioned in Section \ref{subs:modspaces}, (see \cite{NarasimhanSeshadri65} for the original proof and \cite{Donaldson83} for Donaldson's gauge-theoretic proof, which is more in the spirit of the Atiyah \& Bott approach) shows that the space of polystable holomorphic structures are those that are $\mathcal{G}^\C$-equivalent to the minimum of $\YM$. In fact, by considering the Yang-Mills flow, more structure is apparent: it is a consequence of the work of Daskalopoulos in \cite{Daskal92} and R{\aa}de in \cite{Rade92} that the Yang-Mills flow defines a continuous $\mathcal{G}$-equivariant deformation retraction of $\mathcal{A}_{ss}$ onto the minimum of $\YM$ and that the limit is determined by the Seshadri filtration (see \cite[Ch V, Thm 1.15]{kobayashi:1987}) of the initial condition.

This relationship between the algebraic geometry of the Harder-Narasimhan filtration and the analysis of the Yang-Mills flow also has an analog for unstable bundles. Atiyah and Bott show that the Yang-Mills functional achieves a minimum on each stratum $\mathcal{A}_\lambda$, and that this minimising set (call it $\mathcal{C}_\lambda$) is precisely the set of critical points for $\YM$ that are contained in $\mathcal{A}_\lambda$. Moreover, the Morse index at each critical point is the same as the codimension of the corresponding Harder-Narasimhan stratum. The results of \cite{Daskal92} and \cite{Rade92} then show that the Yang-Mills flow defines a $\mathcal{G}$-equivariant deformation retraction of $\mathcal{A}_\lambda$ onto $\mathcal{C}_\lambda$ and that the limit of the flow with initial condition $d_{A_0}$ is isomorphic to the graded object of the Harder-Narasimhan-Seshadri double filtration of $d_{A_0}$. 

Therefore, the problem of studying the $\mathcal{G}$-equivariant cohomology of the space of semistable holomorphic structures is the same as that of studying the $\mathcal{G}$-equivariant cohomology of the minimum of $\YM$, and the inductive formula in terms of the cohomology of the Harder-Narasimhan strata now has a Morse-theoretic analog in terms of the equivariant cohomology of the critical sets for $\YM$.

It is worth mentioning here the relationship with Palais's work. The analytic details of the Morse theory for the Yang-Mills functional are slightly different than the cases studied in \cite{Palais63}. In particular, since the critical sets are infinite-dimensional, then the Palais-Smale condition C fails for the Yang-Mills functional. To recover an analog of condition C, the solution is to instead look at the space of connections modulo the gauge group. Using Uhlenbeck's compactness theorem from \cite{Uhlenbeck82}, Daskalopoulos proves in \cite[Proposition 4.1]{Daskal92} that condition C does hold for $\YM$ on $\mathcal{A} / \mathcal{G}$, and so one should think of the Yang-Mills functional for a bundle over a compact Riemann surface as satisfying a $\mathcal{G}$-equivariant condition C on the space $\mathcal{A}$.

A modified version of condition C is also valid in higher dimensions, where bubbling needs to be taken into account. See for example \cite{Sedlacek82} and \cite[Proposition 4.5]{Taubes84}.

\subsubsection{Higgs bundles}

One would like to extend this picture to spaces of Higgs bundles. We first discuss the case of $G$-Higgs bundles for $G=\GL(n,\C)$.  

\begin{remark}\label{rem:principal-vector-reconcile}
When $G = \GL(n, \C)$ we have $\mathfrak{m}^\C = \mathfrak{gl}(n, \C)$ (see Example \ref{ex:Gc}) and so one can consider the vector bundle $E$ associated to the original principal bundle and think of the space $\mathcal{A}^{0,1}$ of holomorphic structures as an affine space modeled on $\Omega^{0,1}(\End(E))$ and the Higgs field as taking values in $\Omega^0(\End(E) \otimes K)$. Therefore
\begin{equation*}
E_H = \ad(E) \quad \text{and} \quad E_{H^\C}(\mathfrak{m}^\C) = \End(E).
\end{equation*}
Moreover, if we choose the compact real form to be $\mathfrak{u}(n) \subseteq \mathfrak{gl}(n, \C)$, then the involution $\tau : \mathfrak{gl}(n, \C) \rightarrow \mathfrak{gl}(n, \C)$ is given by $\tau(u) = - u^*$.
\end{remark}

Atiyah and Bott's observation that the curvature $F_A$ is a moment map for the action of $\mathcal{G}$ on the space $\mathcal{A} \cong \mathcal{A}^{0,1}$ also extends to the hyperk\"ahler setting of Higgs bundles. The cotangent bundle $T^* \mathcal{A}^{0,1} \cong \mathcal{A}^{0,1} \times \Omega^{1,0}(\End(E))$ has a \emph{hyperk\"ahler structure}, and the action of the gauge group $\mathcal{G}$ has associated moment maps
\begin{align*}
\mu_1(\bar{\partial}_E, \varphi) & = F_A + [\varphi, \varphi^*]  \\
\mu_\C(\bar{\partial}_E, \varphi) = \mu_2 + i \mu_3 & = 2i \bar{\partial}_E \varphi .
\end{align*}

The \emph{hyperk\"ahler moment map} combines these three moment maps using the imaginary quaternions
\begin{equation*}
\mu_{hk} (\bar{\partial}_E, \varphi) = i \mu_1 + j \mu_2 + k \mu_3 \in \Lie(\mathcal{G})^* \otimes_\R \R^3 .
\end{equation*}

In analogy with the Yang-Mills functional (which is the norm-square of the moment map), the \emph{full Yang-Mills-Higgs functional} is defined to be the norm-square of the hyperk\"ahler moment map
\begin{align}\label{eqn:full-YMH-functional}
\begin{split}
\YMH : \mathcal{A}^{0,1} \times \Omega^{1,0}(\End(E)) & \rightarrow \R \\
 (\bar{\partial}_E, \varphi) \mapsto \| \mu_1 \|^2 + \| \mu_2 \|^2 + \| \mu_3 \|^2 & = \| F_A + [\varphi, \varphi^*] \|^2 + 4 \| \bar{\partial}_E \varphi \|^2 .
\end{split}
\end{align} 

\begin{remark}
Recall from Remark \ref{rem:principal-vector-reconcile} that $\mathcal{A}^{0,1} \times \Omega^{1,0}(\End(E)) \cong \mathcal{A}(E_H) \times \Omega^{1,0}(E_{H^\C}(\mathfrak{m}^\C)$, and so the domain of $\YMH$ defined above is consistent with that given in \eqref{eqn:YMH}. Also, since the involution $\tau$ is given by $\tau(u) = - u^*$, then we also have $\| F_A + [\varphi, \varphi^*] \|^2 = \| F_A - [\varphi, \tau(\varphi)] \|^2$.
\end{remark}

The full Yang-Mills-Higgs functional can also be re-written in the following form (compare with \cite[Proposition 2.1]{Bradlow90} for the case where $\varphi \in \Omega^0(E)$).

\begin{lemma}
\begin{multline}
\YMH(\bar{\partial}_A, \varphi) = \| F_A \|^2 + \| [\varphi, \varphi^*] \|^2  + 2 \| \nabla_A \varphi \|^2 + 2 \left< \varphi \circ \Ric, \varphi \right> \\
- 2 \| \partial_A \varphi \|^2 - 2 \| \partial_A^* \varphi \|^2 + 2 \| \bar{\partial}_A \varphi \|^2 .
\end{multline}
\end{lemma}

\begin{proof}
First we expand the Yang-Mills-Higgs functional as
\begin{equation}\label{eqn:expanded-YMH}
\YMH(\bar{\partial}_A, \varphi) = \| F_A \|^2 + \| [\varphi, \varphi^*] \|^2 + 2 \Re \left< F_A, [\varphi, \varphi^*] \right> + 4 \| \bar{\partial}_A \varphi \|^2 .
\end{equation}

Next, recall the Weitzenbock identity for a connection $d_A$ (with covariant derivative denoted by $\nabla_A$) on a bundle $F$ and a one-form $\eta \in \Omega^1(F)$. Let
\begin{equation*}
\mathfrak{R}^A(\eta)_X := \sum_{j=1}^n \left[ R^A(e_j, X), \eta(e_j) \right] ,
\end{equation*}
where $R$ denotes the Riemannian curvature tensor on the base manifold $M$ with respect to the Levi-Civita connection $\nabla$, and let $\Ric$ denote the Ricci tensor of $M$. Then (see \cite[Theorem 3.1]{BourguignonLawson81}) the Weitzenbock identity is
\begin{equation}\label{eqn:weitzenbock}
\Delta_A \eta = \nabla_A^* \nabla_A \eta + \eta \circ \Ric + \mathcal{R}^A(\eta) .
\end{equation}
The term coupling the curvature and the Higgs field in \eqref{eqn:expanded-YMH} can now be written as
\begin{align*}
2 \Re \left< F_A, [\varphi, \varphi^*] \right> & = - 2 \left< \mathcal{R}^A(\varphi), \varphi \right> \\
 & = 2 \left< \nabla_A^* \nabla_A \varphi, \varphi \right> + 2 \left< \varphi \circ \Ric, \varphi \right> - 2 \left<  \Delta_A \varphi, \varphi \right> \\
 & = 2 \| \nabla_A \varphi \|^2 + 2 \left< \varphi \circ \Ric, \varphi \right> - 2 \| d_A \varphi \|^2 - 2 \| d_A^* \varphi \|^2 ,
\end{align*}
where in the second step we use the Weitzenbock identity \eqref{eqn:weitzenbock}. Since $\varphi$ is a $(1,0)$ form, then $d_A^* \varphi = \partial_A^* \varphi$. Decomposition by type means that we also have $\| d_A \varphi \|^2 = \| \bar{\partial}_A \varphi \|^2 + \| \partial_A \varphi \|^2$. Subsitituting this into the formula for the Yang-Mills-Higgs functional gives us
\begin{multline*}
\YMH(\bar{\partial}_A, \varphi) = \| F_A \|^2 + \| [\varphi, \varphi^*] \|^2  + 2 \| \nabla_A \varphi \|^2 + 2 \left< \varphi \circ \Ric, \varphi \right> \\
- 2 \| \partial_A \varphi \|^2 - 2 \| \partial_A^* \varphi \|^2 + 2 \| \bar{\partial}_A \varphi \|^2 ,
\end{multline*}
as required.
\end{proof}

Unfortunately, the full Yang-Mills-Higgs functional described above is not easy to study from the perspective of Morse theory. For example, a complete classification of the critical sets does not exist (although one can classify those satisfying $\bar{\partial}_A \varphi = 0$), and currently there are no theorems on the long-time existence of the gradient flow (again, the difficulties occur when $\bar{\partial}_A \varphi \neq 0$). From this perspective, it is simpler to restrict to the space $\mathcal{B}(\GL(n, \C)) := \{ (\bar{\partial}_E, \varphi) \, : \, \bar{\partial}_E \varphi = 0 \}$ of Higgs bundles, although this introduces additional complications due to singularities in the space $\mathcal{B}(\GL(n, \C))$. Fortunately these complications can be dealt with for cases of low rank; this will be explained later in the section.

On restriction to the space of Higgs bundles, the Yang-Mills-Higgs functional takes on the form
\begin{align}\label{eqn:YMH-functional}
\begin{split}
\YMH : \mathcal{B}(\GL(n, \C)) & \rightarrow \R \\
 (\bar{\partial}_E, \varphi) & \mapsto \| F_A + [\varphi, \varphi^*] \|^2  .
\end{split}
\end{align}

Some of the theory from holomorphic bundles does carry over to $\GL(n,\C)$-Higgs bundles. The Hitchin-Kobayashi correspondence shows that the polystable Higgs bundles are precisely those that are $\mathcal{G}^\C$-equivalent to a minimiser for the Yang-Mills-Higgs functional. There is also an analogous Harder-Narasimhan filtration (defined using $\varphi$-invariant holomorphic sub-bundles) and a corresponding algebraic stratification of the space of Higgs bundles.

The space of Higgs bundles, i.e. $\mathcal{B}(\GL(n,\C))$, may be singular, so \emph{a priori} it is not obvious how to extend the rest of the theory of Atiyah and Bott to study the topology of the space of semistable Higgs bundles. More seriously, the negative eigenspace of the Hessian does not have constant dimension on each connected component of the set of critical points of $\YMH$. As explained in the next few sections, it is possible to get around these difficulties, and we have done this in \cite{DWWW11}, \cite{wentworthwilkin-pairs} and \cite{wentworthwilkin-u21}. Some of the constructions are general, however some have been done by hand for specific cases, so, for now, we have restricted to certain low-rank situations.

The general strategy for extending the Atiyah and Bott theory to the space of Higgs bundles (outlined in more detail in the next few sections) is as follows. Firstly, the analog of the results of Daskalopoulos and R{\aa}de from \cite{Daskal92} and \cite{Rade92} for the Yang-Mills flow also holds for the Yang-Mills-Higgs functional. Secondly, for certain low-rank cases, we are able to show that as the Yang-Mills-Higgs functional passes a critical value, then the topology of the space changes by attaching a certain topological space determined by the negative eigenspace of the Hessian at each critical point. Thirdly, we can compute the change in cohomology as the Yang-Mills-Higgs functional passes a critical value, and thus can compute the cohomology of the moduli space of Higgs bundles for certain low-rank cases.

Finally, it is worth mentioning that the above strategy applies to a much broader class of moduli spaces than just Higgs bundles. For example, the paper \cite{wentworthwilkin-pairs} uses this method to produce new information about the space of rank $2$ stable pairs.

\subsection{Gradient Flow}\label{sect:gradflow}

This section describes the results on the convergence of the gradient flow of the Yang-Mills-Higgs functional on the space of Higgs bundles.

Recall that there is already an algebraic stratification of the space of Higgs bundles. Each Higgs bundle $(\bar{\partial}_A, \varphi)$ has a maximal semistable Higgs sub-bundle (a $\varphi$-invariant holomorphic sub-bundle, whose holomorphic structure and Higgs field are induced from $(\bar{\partial}_A, \varphi)$), and from this one can construct a filtration of $(\bar{\partial}_A, \varphi)$ analogous to the Harder-Narasimhan filtration for holomorphic bundles. The type $\nu$ of the filtration is the vector consisting of the slopes of the quotient bundles counted with multiplicity, and there is a stratification
\begin{equation}\label{eqn:H-N-strata}
\mathcal{B} = \bigcup_{\nu} \mathcal{B}_\nu ,
\end{equation}

\noi (where for simplicity we have dropped $\GL(n,\C)$ from the notation).  The main theorem of \cite{Wilkin08} shows that the Yang-Mills-Higgs flow on $\mathcal{B}$ respects the Harder-Narasimhan filtration in the following sense.

\begin{theorem}\label{thm:YMH-flow-convergence}
The Yang-Mills-Higgs flow with any initial condition $(\bar{\partial}_A, \varphi)$ exists for all time and converges in the smooth topology to a unique limit point, which is a critical point of $\YMH$. Moreover, this critical point is isomorphic to the graded object of the Harder-Narasimhan-Seshadri double filtration of $(\bar{\partial}_A, \varphi)$.
\end{theorem}

Again, it is worth pointing out the connection with the Palais-Smale condition C.
In this context, not only are the critical sets infinite-dimensional, but, even after dividing
by the gauge group, they are non-compact. The necessary compactness ensured by condition C
can, however, be recovered by restricting to complex gauge orbits.  This allows one to prove an analog of the result in \cite{Daskal92} for the Yang-Mills functional. Alternatively, as in \cite{Wilkin08},  one can show that the gradient 
flow converges by applying the Moser iteration technique to obtain improved Sobolev estimates along the 
flow lines, from which Uhlenbeck's compactness theorem gives strong convergence in the appropriate norm.

Each Harder-Narasimhan stratum $\mathcal{B}_\nu$ contains a critical set (corresponding to the minimum of $\YMH$ restricted to $\mathcal{B}_\nu$), and the Harder-Narasimhan type $\nu$ is determined by the value of $F_A + [\varphi, \varphi^*]$ on this critical set. The above theorem implies that there is a Morse stratification 
\begin{equation}\label{eqn:Morse-strata}
\mathcal{B} = \bigcup_{\nu} \mathcal{B}_\nu^{Morse}
\end{equation}
where $(\bar{\partial}_A, \varphi)$ is in the \emph{Morse stratum} $\mathcal{B}_\nu^{Morse}$ if and only if the limit of the flow with initial conditions $(\bar{\partial}_A, \varphi)$ is in the critical set $\mathcal{C}_\nu$ defined earlier.

The methods of \cite{Wilkin08} also show that the gradient flow deformation retraction of each Morse stratum onto the associated critical set is continuous. As a consequence, we see that 
\begin{corollary}\label{cor:algebraic-equals-analytic}
\begin{enumerate}
\item $\mathcal{B}_\nu = \mathcal{B}_\nu^{Morse}$ for all types $\nu$, i.e. the Morse strata are the same as the Harder-Narasimhan strata, and

\item the flow defines a continuous $\mathcal{G}$-equivariant deformation retraction of each stratum $\mathcal{B}_\nu$ onto the corresponding critical set $\mathcal{C}_\nu$.
\end{enumerate}
\end{corollary}


The paper \cite{BiswasWilkin10} extends Theorem \ref{thm:YMH-flow-convergence} to $G$-Higgs bundles whose structure group is any complex reductive group. The remainder of this section shows how to extend this result further to the case where $G$ is a real reductive group as in Section \ref{GHiggs}.  As in Sections \ref{subs:defns} and \ref{subs:modspaces}, we fix a smooth principal $H^\C$ bundle over $X$, denoted $E_{H^\C}$, with topological invariant $d$ and  let $\mathcal{B}_d(G)$ denote the space of Higgs $G$-bundles on $E_{H^\C}$. Let $E_{G^\C} = E_{H^\C} \times_{H^\C} G^\C$ be the bundle associate to $E_{H^\C}$ by the inclusion $H^\C \hookrightarrow G^\C$. This induces an inclusion $i_G : \mathcal{B}_d(G) \hookrightarrow \mathcal{B}_d({G^\C})$.

Given a reduction of structure group of $E_{H^\C}$ from $H^\C$ to $H$, we can define the Yang-Mills-Higgs functional on the space of Higgs $G$-bundles
\begin{equation}\label{eqn:YMH_G}
\YMH_G (\bar{\partial}_A, \varphi) = \| F_A - [\varphi, \tau(\varphi) ] \|^2 .
\end{equation} 

\noi where $\tau$ is the conjugation on $\mathfrak{g}_\C$ that determines the compact Lie algebra $\mathfrak{h} \subset \mathfrak{g}$ (as in \eqref{Higgsequations1}). 

\begin{lemma}\label{lem:real-flow-existence}
The downwards gradient flow of $\YMH_G$ on the space $\mathcal{B}_d(G)$ exists and is unique.
\end{lemma}

\begin{proof}
First consider the following equation (cf. \cite[Section 6]{Simpson88} or \cite[pp7-8]{Donaldson85}) on the space of metrics for the bundle $E_{H^\C}$.
\begin{equation*}
\frac{\partial h}{\partial t} = -2ih*\left( F_{A_0} + \bar{\partial}_{A_0}(h^{-1}(\partial_{A_0} h)) + [\varphi_0, h^{-1} \tau(\varphi_0) h]  \right) + 2i \lambda h .
\end{equation*}

Simpson shows in \cite[Section 6]{Simpson88} that the solution exists and is unique. The same method as in \cite[Section 3.1]{Wilkin08} (see also \cite{Hong01}) shows that this implies existence of a curve $g : \R_{\geq 0} \rightarrow \mathcal{G}_H^\C$ that generates the gradient flow of $\YMH_G$ in the sense that a solution $(\bar{\partial}_{A(t)}, \varphi(t))$ satisfies
\begin{equation*}
(\bar{\partial}_{A(t)}, \varphi(t)) = g(t) \cdot (\bar{\partial}_{A_0}, \varphi_0) .
\end{equation*}

Since $\mathcal{G}_H^\C \subset \mathcal{G}_{G^\C}$ preserves the space $\mathcal{B}_d(G)$, then this gives a unique integral curve for the downwards gradient vector field of $\YMH_G$ on the space $\mathcal{B}_d(G)$.
\end{proof}



Next we show that this integral curve coincides with that of the gradient flow of $\YMH_{G^\C}$ with initial conditions in $\mathcal{B}(G)$. Recall that the Cartan involution $\theta$ on $\mathfrak{g}$ is the restriction of the involution $\tau : \mathfrak{g}_\C \rightarrow \mathfrak{g}_\C$ that determines the compact real form $\mathfrak{h} \oplus i \mathfrak{m} \subset \mathfrak{g}_\C$. For the rest of this section we use $\tau$ to denote both the involution on $\mathfrak{g}_\C$ and its restriction to the Cartan involution on $\mathfrak{g}$.


Let $B$ be the Killing form of $\mathfrak{g}$, and define
\begin{equation*}
B_\tau : \mathfrak{g} \rightarrow \mathfrak{g} , \quad B_\tau(u, v) = -B(u, \tau(v)) . 
\end{equation*}
Since $\tau$ is a Cartan involution then $B_\tau$ is positive definite on the real vector space $\mathfrak{g}$, and therefore extends to a Hermitian inner product on $\mathfrak{g}_\C$. Note that the $\ad$-invariance of the Killing form implies that
\begin{multline}\label{eqn:ad-involution}
B_\tau([u,v], w) = - B([u,v], \tau(w)) = - B(u, [v, \tau(w)]) \\
= B_\tau(u, \tau[v, \tau(w)]) = B_\tau(u, [\tau(v), w]) .
\end{multline}



Using this inner product, we can construct the following inner products on the spaces $\Omega^2(\ad(E_G)) \cong \Lie(\mathcal{G})^*)$ and $\Omega^{0,1}(\ad(E_{G^\C})) \oplus \Omega^{1,0}(\ad(E_{G^\C})) \cong T_{(\bar{\partial}_A, \varphi)} T^* \mathcal{A}^{0,1}(E_{G^\C})$. 
\begin{align*}
\left< \mu, \nu \right> & : = \int_X B_\tau (\mu, * \nu) \\
\left< (a_1, \varphi_1), (a_2, \varphi_2) \right> & := \int_X B_\tau(a_1, *a_2) + \int_X B_\tau(\varphi_1, *\varphi_2) , 
\end{align*}
where (in local coordinates) $\mathcal{B}_\tau$ acts on pairs of $1$-forms by
\begin{equation*}
\mathcal{B}_\tau (f_1 dz + g_1 d \bar{z}, f_2 dz + g_2 d \bar{z}) = - B(f_1, \tau(f_2) ) dz d \bar{z} - B(g_1, \tau(g_2)) d \bar{z} dz ,
\end{equation*}
and on a pair consisting of a $2$-form and a $0$-form by
\begin{equation*}
\mathcal{B}_\tau (f dz d\bar{z}, g) = - B(f, \tau(g)) dz d\bar{z} .
\end{equation*}
Note also that for $\varphi_1 = f_1 dz + g_1 d \bar{z}$ and $\varphi_2 = f_2 dz + g_2 d \bar{z}$ we have
\begin{align*}
\left< [\varphi_1, \varphi_2], \mu \right> & = \int_X B_\tau([\varphi_1, \varphi_2], * \mu) \\
 & = - \int_X B([\varphi_1, \varphi_2], \tau(* \mu) ) \\
 & = - \int_X B( \varphi_1, [\varphi_2, \tau(*\mu)]  ) \\
 & = - \int_X B( \varphi_1, \tau[ \tau(\varphi_2), * \mu] ) \\
 & = \int_X B_\tau(\varphi_1, [\tau(\varphi_2), * \mu] ) \\
 & = - \left< \varphi_1, * [\tau(\varphi_2), * \mu] \right> .
\end{align*}

The definitions above, together with the observation that $\tau$ restricts to the Cartan involution on $\mathfrak{g}$, show that the functional $\YMH_{G^\C}$ on $\mathcal{B}_{G^\C}$ restricts to $\YMH_G$ on $\mathcal{B}_G$. 


Using the calculations above, we see that the derivative of $\YMH_G$ at a point $(\bar{\partial}_A, \varphi) \in \mathcal{B}_G$ is given by
\begin{align*}
d \YMH_G (\delta a, \delta \varphi) & = 2 \Re \left< d_A (\delta a) - [\varphi, \tau(\delta \varphi)] - [\delta \varphi, \tau(\varphi)] , F_A - [\varphi, \tau(\varphi)] \right> \\
 & = 2 \Re \left< \delta a, d_A^*(F_A - [\varphi, \tau(\varphi)] \right> + 2\Re \left< \delta \varphi, *[\varphi, *(F_A - [\varphi, \tau(\varphi)]) ] \right> \\
 & \quad \quad + 2\Re \left< \tau(\delta \varphi), \tau(*[\varphi, *(F_A - [\varphi, \tau(\varphi)]) ] ) \right> \\
 & = 4 \Re \left< (\delta a)^{0,1}, - *\bar{\partial}_A *(F_A - [\varphi, \tau(\varphi)])  \right> \\
 & \quad \quad + 4 \Re \left< \delta \varphi, *[\varphi, *(F_A - [\varphi, \tau(\varphi)]) ] \right> ,
\end{align*}
and so the gradient is
\begin{equation*}
\grad \YMH_G = \left( - i\bar{\partial}_A *(F_A - [\varphi, \tau(\varphi)]), -i [\varphi, *(F_A - [\varphi, \tau(\varphi)]) ]  \right) ,
\end{equation*}
which is the same as that of $\YMH_{G^\C}$ at a point $(\bar{\partial}_A, \varphi) \in \mathcal{B}_G$.

Therefore we have proven the following

\begin{lemma}
Let $\YMH_G$ denote the Yang-Mills-Higgs functional on $\mathcal{B}_G$, and let $\YMH_{G^\C}$ denote the Yang-Mills-Higgs functional on $\mathcal{B}_{G^\C}$. Then
\begin{enumerate}
\item $\YMH_G$ is the restriction of $\YMH_{G^\C}$ to $i_G(\mathcal{B}_G) \subset \mathcal{B}_{G^\C}$, and

\item at a $G$-Higgs bundle $(\bar{\partial}_A, \varphi) \in \mathcal{B}_G$, the gradient of $\YMH_G$ is equal to the gradient of $\YMH_{G^\C}$.
\end{enumerate}
\end{lemma}

Then the gradient flow line of $\YMH_G$ with initial conditions $(\bar{\partial}_{A_0}, \varphi_0) \in \mathcal{B}_G$ from Lemma \ref{lem:real-flow-existence} is an integral curve of the gradient vector field $\YMH_{G^\C}$ with initial conditions $(\bar{\partial}_{A_0}, \varphi_0)$. Since we have uniqueness for the gradient flow of both functionals, then the gradient flow line of $\YMH_G$ must coincide with the gradient flow line of $\YMH_{G^\C}$ on the space $\mathcal{B}_G$. The inclusion $\mathcal{B}_G \hookrightarrow \mathcal{B}_{G^\C}$ is closed, and therefore the limit of the flow of $\YMH_{G^\C}$ is also contained in $\mathcal{B}_G$. In summary, we have shown that
\begin{theorem}\label{newtheorem}
The gradient flow of $\YMH_G$ is the restriction of the gradient flow of $\YMH_{G^\C}$ to the space of $G$-Higgs bundles. Moreover, the gradient flow of $\YMH_G$ converges to a critical point in $\mathcal{B}_G$.
\end{theorem}


\subsection{Cohomology calculations}


First, recall the stratifications \eqref{eqn:H-N-strata} and \eqref{eqn:Morse-strata}, which are the same by Corollary \ref{cor:algebraic-equals-analytic}. If we choose a total ordering compatible with the partial ordering on Harder-Narasimhan type, then we can write this as
\begin{equation}
\mathcal{B} = \bigcup_{j = 0}^\infty \mathcal{B}_j ,
\end{equation}
where $\mathcal{B}_0$ is the semistable stratum, and define
\begin{equation*}
X_d = \bigcup_{j=0}^d \mathcal{B}_j .
\end{equation*}
Then, since the total space $\mathcal{B}$ is contractible, the problem of computing the cohomology of $\mathcal{B}_0 = X_0$ reduces to understanding how the cohomology of $X_d$ relates to that of $X_{d-1}$ for all $d$. Therefore, the goal is to compute $H_\mathcal{G}^*(X_d, X_{d-1})$ and understand the map $b^k$ in the long exact sequence
\begin{equation}\label{eqn:stratified-LES}
\cdots \rightarrow H_\mathcal{G}^k(X_d, X_{d-1}) \stackrel{a^k}{\rightarrow} H_\mathcal{G}^k(X_d) \stackrel{b^k}{\rightarrow} H_\mathcal{G}^k(X_{d-1}) \rightarrow \cdots .
\end{equation}

The first step is to show that $H_\mathcal{G}^*(X_d, X_{d-1})$ can be computed in terms of local data in a neighbourhood of the critical set contained in $\mathcal{B}_d = X_d \setminus X_{d-1}$.

In addition, if the map $a^k$ is injective for all $k$ on each stratum $\mathcal{B}_d$, then the long exact sequence \eqref{eqn:stratified-LES} splits into short exact sequences, and (with coefficients in $\C$) we have
\begin{equation*}
H_\mathcal{G}^k(X_d) \cong H_\mathcal{G}^k(X_{d-1}) \oplus H_\mathcal{G}^k(X_d, X_{d-1}) .
\end{equation*}
In the situations studied by Atiyah \& Bott in \cite{AtiyahBott83} and Kirwan in \cite{Kirwan84}, the Atiyah-Bott lemma \cite[Proposition 13.4]{AtiyahBott83} is the main tool in the proof that $a^k$ is injective. When the ambient space is singular then the situation is more complicated; this is explained below.


Recall that in the Morse-Kirwan case, the dimension of the negative eigenspace of the Hessian is constant on each critical set, and the negative eigenspaces glue together to form a bundle over the critical set, which we call the \emph{negative normal bundle}. First we define the analog of the negative normal bundle for Higgs bundles.

Given a smooth splitting $E = E_1 \oplus \cdots \oplus E_n$, where the sub-bundles are ordered by decreasing slope, consider the following sub-bundle of $\End(E)$
\begin{equation*}
\End(E)_{LT} : = \bigoplus_{i < j} E_i^* \otimes E_j .
\end{equation*}

The critical point equations imply that the eigenvalues of $*(F_A + [\varphi,\varphi^*]) \in \Omega^0(\End(E))$ are locally constant and that the bundle $E$ splits into $\varphi$-invariant sub-bundles corresponding to the eigenspaces. Therefore, at each critical point there is a well-defined subspace $\End(E)_{LT} \subset \End(E)$.

Consider the trivial bundle over the critical set $\eta_d$ with fibres the vector spaces $\Omega^{0,1}(\End(E)) \oplus \Omega^{1,0}(\End(E))$, and let $N$ denote the sub-bundle whose fibres are the subspaces $\Omega^{0,1}(\End(E)_{LT}) \oplus \Omega^{1,0}(\End(E)_{LT})$. Note that $N$ is trivial over any subset of critical points that induce the same smooth splitting of $E$. We define the fibres of the negative normal bundle over a critical point $(\bar{\partial}_A, \varphi)$ to consist of those elements in the fibres of $N$ that are (a) orthogonal to the $\mathcal{G}^\C$ orbit passing through $(\bar{\partial}_A, \varphi)$, and (b) contained in the space of Higgs bundles. More precisely, define the infinitesimal action of $\mathcal{G}^\C$ through $(\bar{\partial}_A, \varphi)$ by
\begin{align*}
\rho_{(\bar{\partial}_A, \varphi)}^\C : \Omega^0(\End(E)) & \rightarrow \Omega^{0,1}(\End(E)) \oplus \Omega^{1,0}(\End(E)) \\
u & \mapsto (\bar{\partial}_A u, [\varphi, u]) ,
\end{align*}
and note that since $(\bar{\partial}_A, \varphi)$ is split at a critical point, then $\rho_{(\bar{\partial}_A, \varphi)}^\C$ restricts to a map $\Omega^0(\End(E)_{LT}) \rightarrow \Omega^{0,1}(\End(E)_{LT}) \oplus \Omega^{1,0}(\End(E)_{LT})$. The \emph{negative normal bundle} is
\begin{equation*}
\nu_d^- := \left\{ (\bar{\partial}_A, \varphi, \delta a, \delta \varphi) \in N \, : \, ( \rho_{(\bar{\partial}_A, \varphi)}^\C )^*(\delta a, \delta \varphi) = 0, \bar{\partial}_{(A+a)}(\varphi + \delta \varphi) = 0 \right\} .
\end{equation*}

\begin{remark}
The term ``negative normal bundle'' is intended intuitively rather than literally, since $\nu_d^-$ is the analog of the negative normal bundle for a Morse-Kirwan function on a smooth space. When the ambient space is singular then $\nu_d^-$ as defined above may not be a bundle. As we will see below in Example \ref{ex:index-jumps}, the fibres may not always be vector spaces, and even if the fibres are vector spaces then they may jump in dimension.
\end{remark}

We also define the associated space
\begin{equation*}
\nu_d' := \left\{ (\bar{\partial}_A, \varphi, \delta a, \delta \varphi) \in \nu_d^- \, : \, (\delta a, \delta \varphi) \neq 0 \right\} .
\end{equation*}

The following theorem from \cite{DWWW11}, \cite{wentworthwilkin-pairs} and \cite{wentworthwilkin-u21} shows that $H_\mathcal{G}^*(X_d, X_{d-1})$ can be computed in terms of the local data of the negative normal bundle. The analogous theorem for non-degenerate Morse theory is contained in \cite[Chapter 3]{Milnor63} and \cite[Section 12]{Palais63}, for Morse-Bott theory it was proven by Bott in \cite[p250]{Bott54}, and for Morse-Kirwan functions it was proven by Kirwan in \cite[Chapter 10]{Kirwan84}.

\begin{theorem}
For the cases of rank $2$ Higgs bundles, $U(2,1)$ Higgs bundles and rank $2$ stable pairs, we have
\begin{equation*}
H_\mathcal{G}^*(X_d, X_{d-1}) \cong H_\mathcal{G}^*(\nu_d^-, \nu_d') .
\end{equation*}
\end{theorem}

As noted earlier, even if the fibres of $\nu_d^-$ are vector spaces, their dimension may be non-constant on each connected component of the critical set of $\YMH$, and so the problem of computing $H_\mathcal{G}^*(\nu_d^-, \nu_d')$ is more complicated than the methods of Atiyah \& Bott and Kirwan from \cite{AtiyahBott83} and \cite{Kirwan84}. The following example illustrates how the fibres may jump in dimension.

\begin{example}\label{ex:index-jumps}
Let $E$ be a rank $2$ bundle of degree $d$, and consider the critical points of $\YMH$ where $E$ splits into line bundles $L_1 \oplus L_2$, with $\deg L_1 = \ell$. This forms a connected component of the set of critical points, denoted $\eta_\ell$. Modulo the action of the unitary gauge group we can fix a smooth splitting $E = L_1 \oplus L_2$ (cf. \cite[Section 7]{AtiyahBott83} and \cite[Corollary 4.17]{Wilkin08}), and with respect to this splitting $\End(E)_{LT} = L_1^* L_2$. Then, for $(\delta a, \delta \varphi) \in \Omega^{0,1}(\End(E)_{LT}) \oplus \Omega^{1,0}(\End(E)_{LT})$, the conditions $(\delta a, \delta \varphi) \in \ker (\rho_{(\bar{\partial}_A, \varphi)}^\C)^*$ and $\bar{\partial}_{A+a}(\varphi+\delta \varphi) = 0$ become
\begin{equation*}
\bar{\partial}_A^* (\delta a) - \bar{*}[\varphi, \bar{*}(\delta \varphi)] = 0, \quad \bar{\partial}_A (\delta \varphi) + [\delta a, \varphi] = 0 .
\end{equation*}
Note that $[\delta a, \delta \varphi] = 0$, since $\End(E)_{LT} = L_1^* L_2$. Therefore the above equations are linear in $(\delta a, \delta \varphi)$, and so the space of solutions is a vector space.

When $\varphi = 0$, the equations reduce to
\begin{equation*}
\bar{\partial}_A^* (\delta a) = 0, \quad \bar{\partial}_A (\delta \varphi) = 0 ,
\end{equation*}
and so the space of solutions is isomorphic to $H^{0,1}(L_1^* L_2) \oplus H^{1,0}(L_1^* L_2)$. By Riemann-Roch, the dimension is
\begin{equation*}
2 \ell - \deg (E) + g-1 + h^0(L_1^* L_2 \otimes K) .
\end{equation*}
When $0 \leq \deg(L_1^* L_2 \otimes K) \leq 2g-2$ then $h^0(L_1^* L_2 \otimes K)$ is not constant  with respect to the holomorphic structure on $L_1^* L_2 \otimes K$. Therefore the dimension of the fibres of $\nu_d^-$ is non-constant.

When the Higgs field is non-zero, there is an analogous picture obtained by replacing the cohomology groups $H^{0,1}(L_1^* L_2) \oplus H^{1,0}(L_1^* L_2)$ with the harmonic forms associated to the middle term in the following deformation complex
\begin{equation*}
\Omega^0(L_1^*L_2) \stackrel{\rho_{(\bar{\partial}_A, \varphi)}^\C}{\longrightarrow} \Omega^{0,1}(L_1^* L_2) \oplus \Omega^{1,0}(L_1^* L_2) \stackrel{d \mu_\C}{\longrightarrow} \Omega^{1,1}(L_1^* L_2) .
\end{equation*}
\end{example}

As noted above, for rank $2$ we always have $[\delta a, \delta \varphi] = 0$ for $(\delta a, \delta \varphi) \in \Omega^{0,1}(\End(E)_{LT}) \oplus \Omega^{1,0}(\End(E)_{LT})$, and so it is sufficient to study the linearised equations. For higher rank, one needs to replace the linearisation of the complex moment map $d \mu_\C$ with the map $(\delta a, \delta \varphi) \mapsto \bar{\partial}_{A+\delta a} (\varphi + \delta \varphi)$. In the more complicated cases studied in \cite{wentworthwilkin-pairs} and \cite{wentworthwilkin-u21} this becomes a consideration.



Nevertheless, it is still possible to compute $H_\mathcal{G}^*(\nu_d^-, \nu_d')$ in this case (and for other low-rank cases). The observation that allows us to do this is that the jumps in dimension described above are related to the jumps in dimension of the fibres of the Abel-Jacobi map $S^d X \rightarrow \Jac_d(X)$. 

The strategy of the calculation from \cite{DWWW11} is as follows. First, find a $\mathcal{G}$-equivariant deformation retract of the pair $(\nu_d^-, \nu_d')$ to the subspace with Higgs field zero. Next, introduce the space
\begin{equation*}
\nu_d'' = \left\{ (\bar{\partial}_A, \varphi, \delta a, \delta \varphi) \in \nu_d^- \, : \, \varphi = 0, \delta a \neq 0 \right\} \subseteq \nu_d' .
\end{equation*}

In \cite{DWWW11}, we compute
\begin{align*}
H_\mathcal{G}^*(\nu_d^-, \nu_d'') & \cong H^{*-2\ell +\deg (E) -(g-1)}(\Jac(X) \times \Jac(X)) \otimes H^*(BU(1))^{\otimes 2} \\
H_\mathcal{G}^*(\nu_d', \nu_d'') & \cong H^{*-2\ell +\deg (E) -(g-1)}(S^{2g-2+\deg(E) - 2 \ell_1} X \times \Jac(X)) \otimes H^*(BU(1)).
\end{align*}
Without repeating the details of these calculations, it is still worth noting that both isomorphisms follow from reducing to the Thom isomorphism. In the first case this is straightforward, and in the second case this follows by reducing to a vector bundle over a subset of the negative normal directions.

These computations tell us about the long exact sequence \eqref{eqn:stratified-LES} via the following commutative diagram.
\begin{equation}\label{eqn:rank2-diagram}
\xymatrix{
 & \vdots \ar[d] \\
\cdots \ar[r] & H_\mathcal{G}^k(X_d, X_{d-1}) \ar[r] \ar[d]^\cong & H_\mathcal{G}^k(X_d) \ar[r] \ar[d] & H_\mathcal{G}^k(X_{d-1}) \ar[r] \ar[d] & \cdots \\
 \cdots \ar[r] & H_\mathcal{G}^k(\nu_d, \nu_d') \ar[r] \ar[d] & H_\mathcal{G}^k(\nu_d) \ar[r] & H_\mathcal{G}^k(\nu_d') \ar[r] & \cdots \\
  & H_\mathcal{G}^k(\nu_d, \nu_d'') \ar[ur]^{\xi^k} \ar[d]^{\gamma^k} \\
  & H_\mathcal{G}^k(\nu_d', \nu_d'') \ar[d] \\
  & \vdots 
}
\end{equation}

\begin{lemma}
For non-fixed determinant rank $2$ Higgs bundles, the map $\gamma^k$ is surjective for all $k$, and so the vertical long exact sequence splits into short exact sequences.
\end{lemma}

The Atiyah-Bott lemma shows that the diagonal map $\xi^k$ is injective, and so a diagram chase leads to
\begin{corollary}\label{cor:calculations}
\begin{enumerate}
\item The map $b^k : H_\mathcal{G}^k(X_d) \rightarrow H_\mathcal{G}^k(X_{d-1})$ is surjective, and

\item $P_t^\mathcal{G}(X_d) - P_t^\mathcal{G}(X_{d-1}) = P_t^\mathcal{G}(\nu_d, \nu_d'') - P_t^\mathcal{G}(\nu_d', \nu_d'')$ .
\end{enumerate}
\end{corollary}

The calculations for other examples such as stable pairs and $U(2,1)$ Higgs bundles follow a similar strategy: the idea is to further decompose the negative normal bundle and then compute the cohomology of the pair $(X_d, X_{d-1})$ in terms of on the smaller pieces in the decomposition. See \cite{wentworthwilkin-pairs} and \cite{wentworthwilkin-u21} for more details.


In the fixed determinant case the situation is not quite as simple. The analogous map $\gamma^k$ is not always surjective, and so the vertical long exact sequence in \eqref{eqn:rank2-diagram} does not split. As a consequence, a diagram chase of \eqref{eqn:rank2-diagram} only gives us
\begin{equation*}
P_t^\mathcal{G}(\nu_d) - P_t^\mathcal{G}(\nu_d') = P_t^\mathcal{G}(\nu_d, \nu_d'') - P_t^\mathcal{G}(\nu_d', \nu_d'') ,
\end{equation*}
however we can still recover the second part of Corollary \ref{cor:calculations} in two different ways. 


The first is to avoid using the negative normal bundle in the diagram \eqref{eqn:rank2-diagram}, and instead do all of the calculations in terms of the spaces $X_d$. This is more in the spirit of the original Atiyah and Bott approach. The intermediate space is denoted $X_d''$ (see \cite[Section 3.1]{DWWW11} for the precise definition), and there are isomorphisms 
\begin{equation*}
H_\mathcal{G}^*(X_d, X_d'') \cong H_\mathcal{G}^*(\nu_d, \nu_d''), \quad H_\mathcal{G}^*(X_{d-1}, X_d'') \cong H_\mathcal{G}^*(\nu_d', \nu_d'') .
\end{equation*}
The diagram becomes
\begin{equation}\label{eqn:rank2-stratified-diagram}
\xymatrix{
 & \vdots \ar[d] \\
\cdots \ar[r] & H_\mathcal{G}^k(X_d, X_{d-1}) \ar[r] \ar[d]^\cong & H_\mathcal{G}^k(X_d) \ar[r] & H_\mathcal{G}^k(X_{d-1}) \ar[r] & \cdots \\
  & H_\mathcal{G}^k(X_d, X_d'') \ar[ur]^{\xi^k} \ar[d]^{\gamma^k} \\
  & H_\mathcal{G}^k(X_{d-1}, X_d'') \ar[d] \\
  & \vdots 
}
\end{equation}

Again, the diagonal map $\xi^k$ is injective. An analogous diagram chase to the one above gives us $P_t^\mathcal{G}(X_d) - P_t^\mathcal{G}(X_{d-1}) = P_t^\mathcal{G}(X_d, X_d'') - P_t^\mathcal{G}(X_{d-1}, X_d'')$, even though the vertical long exact sequence in \eqref{eqn:rank2-stratified-diagram} does not split.


The second approach is to use the action of the finite group $\Gamma_2 := H^1(X, \mathbb{Z} / 2\mathbb{Z})$ corresponding to the $2$-torsion points on the Jacobian, which was originally defined by Hitchin in \cite[Section 7]{Hitchin87}. This acts on the cohomology groups in the diagram \eqref{eqn:rank2-diagram}, and the representation splits into pieces according to whether the action is trivial or not. In \cite[Section 4.2]{DWWW11}, we show that the maps in the diagram \eqref{eqn:rank2-diagram} respect this splitting, and that we can recover the result
\begin{equation*}
P_t^\mathcal{G}(X_d) - P_t^\mathcal{G}(X_{d-1}) = P_t^\mathcal{G}(\nu_d, \nu_d'') - P_t^\mathcal{G}(\nu_d', \nu_d'') .
\end{equation*}
As an additional consequence of this method, we also prove 
\begin{theorem}[Theorem 4.13 of \cite{DWWW11}]
The $\Gamma_2$-invariant Kirwan map surjects onto the $\Gamma_2$-invariant part of the cohomology of $\mathcal{B}_0^{ss}(2, d)$ for any $d$.
\end{theorem}

\section{The Hitchin function}\label{sect:hitchinfn}

We now examine the function mentioned in \eqref{Hitchinf}.  Fix a smooth principal $H$-bundle, $E_H$, with topological invariant $d$ and use the description of $\mathcal{M}_d(G)$ in which points are represented by gauge orbits of points $(A,\varphi)\in\mathcal{A}\times\Omega^{1,0}(E_H(\mclie))$.  The function is then

\begin{align}
f:&\mathcal{M}_d(G)\rightarrow\R\nonumber\\
&[A,\varphi]\mapsto\ ||\varphi ||_{L^2}^2 , \label{hitchinfunction}
\end{align}
where the $L^2$-norm of the Higgs field is computed using the (fixed) metric determined by the Ad-invariant bilinear form on $G$ and the metric on $X$. If we regard points in $\mathcal{M}_d(G)$ as isomorphism classes of holomorphic pairs $(E_{\HC},\varphi)$ or, equivalently, as complex gauge orbits in $\mathcal{A}^{0,1}(E_{\HC})\times\Omega^{1,0}(E_H(\mclie)$, then the function is defined by

\begin{equation}
f([E_{\HC},\varphi])=||g(\varphi)||^2 ,
\end{equation}

\noi where $g\in\mathcal{G}^{\C}$ is the element (unique up to real gauge transformations) such that $g(E_{\HC},\varphi)$ lies in the the zero set of the moment map $\Psi$.

\begin{theorem}The map
\begin{align}
S^1\times\mathcal{M}_d(G)\rightarrow\mathcal{M}_d(G)\nonumber\\
e^{i\theta}[E,\varphi]=[E,e^{i\theta}\varphi]\label{S1action}
\end{align}
\noi defines an action which is Hamiltonian on the smooth locus and for which the above function $f$ is the moment map. 
\end{theorem}

\noi It follows immediately by a theorem of Frankel for $S^1$-actions on K\"ahler manifolds (\cite{frankel:1959})that:

\begin{theorem}
The critical points of $f$ are given by the fixed points of the $S^1$ action.  Moreover,  $f$ is a perfect non-degenerate Morse-Bott function on the smooth locus of $\mathcal{M}_d(G)$.
\end{theorem}

\noi Furthermore, if $M$ is any such K\"ahler manifold with an $S^1$-action and $p\in M$ is a smooth critical point of the moment map, then there are two linear maps on the tangent space $T_pM$, namely

\begin{enumerate}
\item the Hessian of the moment map, $H_f$, and
\item the infinitesimal $S^1$-action, $H_S$
\end{enumerate}

\noi and these are related by

\begin{equation}\label{Hf=Hs}
H_f=-iH_S .
\end{equation}

\noi It follows from \eqref{Hf=Hs} that

\begin{proposition}\label{wtspaces} At a critical point $p\in M$ for the $S^1$-moment map $f$, the decomposition of $T_pM$ into eigenspaces for $H_f$ coincides with the weight-space decomposition for the $S^1$-action and that the $\lambda$-eigenspace for $H_f$ is the $-i\lambda$-weight space for the $S^1$-action.
\end{proposition}

\noi In cases where $\mathcal{M}_d(G)$ is smooth, it follows that the Poincare polynomial is given by

\begin{equation}\label{poincare}
P_t(\mathcal{M}_d(G))=\Sigma_Nt^{\lambda_N}P_t(N)
\end{equation}

\noi where the sum is over the critical submanifolds of $f$ and $\lambda_N$ is the index of critical submanifold $N$, however in order to use formula \eqref{poincare} one has to
\begin{itemize}
\item identify the fixed points of the $S^1$-action and hence the critical submanifolds $N$, 
\item compute the Morse indices $\lambda_N$, and
\item compute the Poincare polynomials $P_t(N)$.
\end{itemize}

\noi In Section \ref{subs:best} we discuss some cases where all these steps can be carried out. Even in cases where this is not possible, either because the moduli space is not smooth or because the requisite properties for the critical submanifolds are not known, the Hitchin function $f$ retains the following useful features (first shown by Hitchin in \cite{Hitchin87}).

\begin{theorem}\label{th:proper}The function $f:\mathcal{M}_d(G)\rightarrow\R$ defined as in \eqref{hitchinfunction} is a proper map and hence attains a minimum on each connected component.
\end{theorem}

\begin{proof}(sketch)  Recall that points in $\mathcal{M}_d(G)$ are represented by pairs $(A,\varphi)\in\mathcal{A}\times\Omega^{1,0}(E_H(\mclie)$ which satisfy  equations \eqref{Higgsequations1} and \eqref{Higgsequations2}. Applying a Weitzenbock formula using the connection $A$ and the metric on $X$ thus leads to the estimate

\begin{align}
0\le||d_A^{1,0}\varphi||_{L^2}^2&= -<F_A\ ,[\varphi,\tau(\varphi)]>+(2g-2)||\varphi||_{L^2}^2\\
&= -||[\varphi,\tau(\varphi)]||^2_{L^2}+(2g-2)||\varphi||_{L^2}^2 .
\end{align}

\noi where $d_A^{1,0}$ denotes the holomorphic part of the covariant derivative and the pairing $<\ ,\ >$ uses the Ad-invariant bilinear form on $\lieg$ and the inner product on forms of type $(1,1)$.  

Thus 

\begin{equation}
||F_A||^2_{L^2}=||[\varphi,\tau(\varphi) ]||^2_{L^2}\le (2g-2) f([A,\varphi])
\end{equation}

\noi Since the connections are on a bundle over a compact Riemann surface, the result now follows from a compactness result by Uhlenbeck (\cite{Uhlenbeck82}) for connections with $L^2$-bounds on curvature.
\end{proof}

\noi In order to count the connected components of $\mathcal{M}_d(G)$ it is thus enough to

\begin{itemize}
\item identify the fixed points of the $S^1$-action and hence the critical submanifolds $N$, 
\item show that each minimal submanifold $N$ is connected
\end{itemize}

\noi In Section \ref{subs:goodenough} we discuss some cases where these steps can be carried out.

\subsection{Critical points}\label{subs:cpts}

Denote the infinitesimal generator of the $S^1$-action \eqref{S1action} by the vector field $X$ on $\mathcal{M}_d(G)$. Since $f$ is a moment map with respect to the symplectic form defined by a K\"ahler metric, we have for any other vector field $Y$

\begin{equation}
df(Y)= -k(Y,JX)
\end{equation}

\noi where $J$ denotes the involution which defines the complex structure and the pairing $k$ denotes the K\"ahler metric. In particular, a point $[A,\varphi]\in\mathcal{M}_d(G)$ is a critical point of $f$ if and only if
$X_{[A,\varphi]}=0$, i.e. 

\begin{equation}
\frac{d}{dt}[A,e^{it}\varphi]|_{t=0}=0
\end{equation}

\noi There are two types of critical points that can occur:

{\bf Type 1: [$\varphi=0$]} Clearly any points at which $\varphi=0$ are fixed points of the $S^1$-action and hence critical points of $f$.  Moreover, since $f=0$ at these points, they are in fact global minima.

{\bf Type 2: [$\varphi\ne0$]} Minima with $\varphi\ne 0$ can occur because the points in $\mathcal{M}_d(G)$ are {\it isomorphism classes} of Higgs bundles or, equivalently {\it gauge orbits} of pairs in the configuration space.  The condition for a point represented by the Higgs bundle $(A,\varphi)$ to be a fixed point is that there is a 1-parameter family of gauge transformations, say $g_{\theta}$, such that

\begin{equation}\label{thetane0}
(A, e^{i\theta}\varphi)=g_{\theta}(A,\varphi)
\end{equation}

\noi for all $\theta$.  Let $\Psi$ be the infinitesmal action of $g_{\theta}$, i.e. $(d^A(\Psi), [\Psi, \varphi]) = \left. \frac{d}{d \theta} \right|_{\theta = 0} g_\theta (A, \varphi)$.  Then \eqref{thetane0} implies 

\begin{align}\label{thetaonphi}
\begin{split}
d^A(\Psi)=0,\\
[\Psi, \varphi] = i\varphi\ .
\end{split}
\end{align}

\begin{example}\label{ExGln} If $G=\GL(n,\C)$ and we use the standard representation to replace $E_{\HC}$ by a rank $n$ vector bundle $V$ then the first condition in \eqref{thetaonphi} implies that $\Psi$ decomposes $V$ into a direct sum of eigenbundles, say

\begin{equation}\label{GlHodge1} 
V=\bigoplus_{\lambda} V_{\lambda} \ ,
\end{equation}

\noi where the eigenvalue on $V_{\lambda}$ is $i\lambda$.  The second condition then says that 

\begin{equation}\label{PsionE}
\varphi:V_{\lambda}\rightarrow V_{\lambda+1}\otimes K
\end{equation}

\noi In other words, the Higgs bundle $(V,\varphi)$ defines a variation of Hodge structure (as defined, for example, in \cite{Simpson92}).
\end{example}

In general, the infinitesimal gauge transformation $\Psi$ defines eigenbundle decompositions 

\begin{align}\label{decomp}
E_{H}(\liehc)=\bigoplus_{\mu} (E_{H}(\liehc))_{\mu}\\
E_H(\mclie)=\bigoplus_{\nu} (E_{H}(\mclie))_{\nu}
\end{align}

\noi where the eigenvalues on the summands are $i\mu$ and $i\nu$ respectively. The second condition in \eqref{thetaonphi} implies that

\begin{align}\label{Hodge2}
ad(\varphi):(E_{H}(\liehc))_{\mu}\rightarrow (E_{H}(\mclie))_{\mu+1}\otimes K\\
ad(\varphi):(E_{H}(\mclie))_{\nu}\rightarrow (E_{H}(\liehc))_{\nu+1}\otimes K
\end{align}

\noi where the maps are defined by the adjoint action.  In cases where there is a `natural' or `standard' representation of $\HC$ so that $E_{\HC}$ can be replaced by one or more vector bundles, the fixed points of the $S^1$-action have a Hodge bundle structure and the eigenbundles of  $E_{H}(\liehc)$ and $E_{H}(\mclie)$ can be described in terms of the Hodge bundle summands. 

\begin{example} For $G=\GL(n,\C)$ where, as in Example \ref{ExGln}, we can describe the Higgs bundles as pairs $(V,\varphi)$ in which $V$ is a rank $n$ vector bundle,  the Hodge bundles are as in \eqref{GlHodge1} and \eqref{Hodge2}.  In this case $\liehc=\mclie=\mathfrak{gl}(n,\C)$ and $E_{H}(\liehc)=E_{H}(\mclie)=\End(V)$.  The eigenbundles in the decomposition of $\End(V)$ are given by

\begin{equation}
\End(V)_{\lambda}=\bigoplus_{\lambda=\nu-\mu} \Hom(V_{\mu},V_{\nu})
\end{equation}

\noi i.e. the eigenvalues are of the form $\mu-\nu$ where $\mu,\nu$ are eigenvalues in the Hodge decomposition of $V$.
\end{example}

\begin{example}\label{Ex;Sp2nR}\footnotemark\footnotetext{See \cite{garcia-gothen-mundet:2008} for details}For $G=\Sp(2n,\R)$ we have, as in Table 1,  $H=\U(n)$, so $E_{\HC}$ is a principal $\GL(n,\C)$-bundle. If we let $V$ be the rank $n$ vector bundle given by the standard representation, then 

\begin{align}\label{SpIso}
E_{H}(\liehc)&=\End(V)\nonumber\\
E_H(\mclie)&=Sym^2(V)\oplus Sym^2(V^*)
\end{align}

\noi  so, as in Table 2,  we can describe $\Sp(2n,\R)$-Higgs bundles as triples $(V,\beta,\gamma)$ where 

\begin{equation}
\beta\in H^0(Sym^2(V)\otimes K)\quad and\quad \gamma\in H^0(Sym^2(V^*)\otimes K)\ .
\end{equation}

\noi At the fixed points of the $S^1$-action the infinitesimal generator of the gauge transformation, i.e. the $\Psi$ in \eqref{thetaonphi}, yields an eigenbundle decomposition as in \eqref{GlHodge1} but the effect of the second condition in \eqref{thetaonphi} is not the same as in  \eqref{PsionE}.  For $\Sp(2n,\R)$-Higgs bundles the condition is equivalent to the conditions

\begin{align}
\Psi\beta-\beta\Psi^*&=i\beta\\
\Psi^*\gamma-\gamma\Psi&=i\gamma
\end{align}

\noi and hence we get that 

\begin{equation}\label{SpHodge}
\beta:V^*_{\lambda}\rightarrow V_{\lambda+1}\otimes K\quad\mathrm{and}\quad \gamma:V_{\lambda}\rightarrow V^*_{-\lambda-1}\otimes K
\end{equation}

\noi where $V=\bigoplus_{\lambda} V_{\lambda} $ is the eigenbundle decomposition of $V$.  Thus:

\begin{definition}\label{defn:SpHodge}  The Hodge bundles corresponding to the $S^1$-fixed points on the $\Sp(2n,\R)$-Higgs bundle moduli spaces are of the form $(V,\beta,\gamma)$ with $V=\bigoplus_{\lambda} V_{\lambda} $ and $\beta,\gamma$ satisfying \eqref{SpHodge}.
\end{definition}

\noi It follows from \eqref{SpIso} that the eigenbundle decompositions for $E_{H}(\liehc)$ and $E_H(\mclie)$ are given by

\begin{align}
\End(V)_{\lambda}&=\bigoplus_{\lambda=\mu-\nu}\Hom(V_{\mu},V_{\nu})\\
Sym^2(V)_{\lambda+1}&=\bigoplus_{\mu+\nu=\lambda+1}V_{\mu}\otimes V_{\nu}\oplus Sym^2(V_{\frac{\lambda+1}{2}})\\
Sym^2(V^*)_{\lambda+1}&=\bigoplus_{-\mu-\nu=\lambda+1}V^*_{\mu}\otimes V^*_{\nu}\oplus Sym^2(V^*_{\frac{\lambda+1}{2}})
\end{align}

\end{example}

\subsection{Morse indices and a criterion for the minima}\label{subs:min}

Having identified the critical points of the function $f$, the next step in the Morse theory program requires computation of Morse indices, and this in turn requires an understanding of the tangent spaces to the Higgs bundle moduli spaces.  The main ideas go back to Hitchin's papers (see \cite{Hitchin92}):

The space of infinitesimal deformations of a $G$-Higgs bundle $(E_{H},\varphi)$ can be identified as the first hypercohomology of a deformation complex 

\begin{equation}\label{defcx}
C^{\bullet}(E_{H},\varphi) : E_H(\liehc)\stackrel{ad(\varphi)}\longrightarrow E_H(\mclie)\otimes K
\end{equation}

\noi  If $(E_{H},\varphi)$ represents a smooth point in the moduli space then the first hypercohomology of the deformation complex  is canonically isomorphic to the tangent space at the point, i.e.

\begin{equation}
\mathbb{H}^1(C^{\bullet}(E_{H},\varphi)) \simeq T_{[E_{H},\varphi]}\mathcal{M}_d(G)\ .
\end{equation}

\noi The decompositions \eqref{decomp}, together with \eqref{Hodge2}, decompose the deformation complex as 

\begin{equation}\label{defcxdecomp}
C^{\bullet}(E_{H},\varphi)=\bigoplus_{\mu} C^{\bullet}(E_{H},\varphi)_{\mu}
\end{equation}

\noi where

\begin{equation}\label{defcxmu}
C^{\bullet}(E_{H},\varphi)_{\mu} : E_H(\liehc)_{\mu}\stackrel{ad(\varphi)}\longrightarrow E_H(\mclie)_{\mu+1}\otimes K
\end{equation}

\noi At smooth critical points of $f$ the hypercohomology group $\mathbb{H}^1(C^{\bullet}(E_{H},\varphi)_{\mu} )$ is isomorphic to the $-\mu$ eigenspace for the Hessian $H_f$.

The above results allow Morse indices to be calculated and in particular yield the following criterion for a smooth critical point to be a minimum (see \cite{Hitchin92}):

\begin{proposition}\label{smoothmin} Let $(E_{\HC},\varphi)$ represent a fixed point of the $S^1$-action on the smooth locus of a moduli space of polystable $G$-Higgs bundles $\mathcal{M}_d(G)$. Then the point is a local minimum if and only if 

\begin{equation}
\mathbb{H}^1(C^{\bullet}(E_{H},\varphi)_{\mu} )=0\quad \forall\ k>0\ .
\end{equation}
\end{proposition}

\noi  In practice, i.e. when working out the details for specific examples, the vanishing of  $\mathbb{H}^1(C^{\bullet}(E_{H},\varphi)_{\mu} )$ is determined using a criterion first formulated in \cite{Hitchin92}, namely 

\begin{proposition} Under the same assumptions as in Proposition \ref{smoothmin}, the hypercohomolgy $\mathbb{H}^1(C^{\bullet}(E_{H},\varphi)_{\mu} )$ vanishes if and only if

\begin{equation}\label{hitchincriterion}
ad(\varphi):(E_{H}(\liehc))_{\mu}\rightarrow (E_{H}(\mclie))_{\mu+1}\otimes K
\end{equation}

\noi is an isomorphism.
\end{proposition}

In cases where the moduli space $\mathcal{M}_d(G)$ has singularities, it can happen that fixed points of the $S^1$-action of the function $f$ lie in the singular locus.  To determine whether the resulting Hodge bundles are local minima for the function $f$, one must construct families of deformations by hand and examine the behavior of the function along the deformations.  In Section \ref{subs:goodenough} we illustrate these ideas with some specific examples.

\subsection{Best-case scenarios}\label{subs:best}   In this section we discuss some situations in which the full Morse theory program can be carried out using the Hitchin function $f$.  

As noted in Section \ref{subs:modspaces}, for $G=\GL(n,\C)$, the moduli spaces $\mathcal{M}_d(G)$ are sometimes smooth.   The $\GL(n,\C)$-Higgs bundles can be represented by pairs $(V,\varphi)$ where $V$ is a rank $n$ holomorphic bundle, the Higgs field $\varphi:V\rightarrow V\otimes K$, and the topological class $d$ is the degree of $V$. In order for such Higgs bundles to be polystable but not stable, the bundle $V$ has to split as a direct sum of lower rank vector bundles with the same slope as $V$.  This cannot happen if the rank and degree of $V$ are coprime, i.e. $(n,d)=1$. In \cite{Hitchin87} Hitchin analyzed the case of $n=2$ and odd degree and in \cite{gothen94} Gothen examined the rank three case. 

Fix $d$ such that $(3,d)=1$ and let $\mathcal{M}_d$ denote the moduli space of stable $\GL(3,\C)$-Higgs bundles for which the underlying bundle has degree $d$.  Let $f:\mathcal{M}_d\rightarrow\R$ be the Hitchin function as defined in \eqref{hitchinfunction}.  The critical points of $f$, or equivalently the fixed points of the $S^1$-action, can be of either type discussed in Section \ref {subs:cpts}.  

The critical points with $\varphi=0$ are zeros of the non-negative function $f$ and hence clearly global  minima.  Moreover, if $\mathcal{N}_0$ denotes the locus of such minima, then we can identify

\begin{equation}
\mathcal{N}_0\simeq M(3,d)
\end{equation}

\noi where $M(3,d)$ denotes the moduli space of stable bundles of rank 3 and degree $d$.  The Morse index of $\mathcal{N}_0$ is zero (since the critical points are global minima) and the Poincare polynomial for $M(3,d)$ is known by the work of Desale and Ramanan (\cite{desaleRamanan}).  We can thus compute the contribution of $\mathcal{N}_0$ to the Poincare polynomial of $\mathcal{M}_d$.

The critical points with $\varphi\ne 0$ are described as follows. As explained in Example \ref{ExGln}  the corresponding fixed points of the $S^1$-action are represented by Hodge bundles, say $(V,\varphi)$. Since $n=3$, the possibilities are limited to the three types in Table 3. The possibilities within this set of types are constrained by the topological type of $V$, i.e.
\begin{equation}
\deg(L)+\deg(V_i)=\deg(L_1)+\deg(L_2)+\deg(L_3)=d\ .
\end{equation}

\noi Moreover,  the Hodge bundles must represent points in the moduli space $\mathcal{M}_d$.  They must thus represent {\it stable} Higgs bundles or, equivalently, must support solutions to the Higgs bundle equation \eqref{Higgsequations1}.  This puts further constraints on the degrees of the summands. For example (see \cite{gothen94}) for the Hodge bundles of type $(1,2)$ the constraint is that

\begin{equation}
3d < \deg(L) < \frac{1}{3}d + g- 1.
\end{equation}

\begin{table}
\begin{tabular}{|c|c|c|c|}
\hline
type& $V$& ranks & $\varphi$\\
\hline
$(1,2)$ & $L\oplus V_2$ & $\rank(L)=1,\rank(V_2)=2$ & $\begin{pmatrix}0&0\\\varphi&0\end{pmatrix}$\\
\hline
$(2,1)$ & $V_1\oplus L$ &  $\rank(V_1)=2, \rank(L)=1$ & $\begin{pmatrix}0&0\\\varphi&0\end{pmatrix}$\\
\hline
$(1,1,1)$& $L_1\oplus L_2\oplus L_3$ & $\rank(L_i)=1$ & $\begin{pmatrix}0&0&0\\\varphi_1&0&0\\0&\varphi_2&0\end{pmatrix}$\\
\hline
\end{tabular}
\caption{\small{Hodge bundles defining critical points of the $S^1$-action on the moduli space of $\SL(3,\C)$-Higgs bundles}}
\end{table}

Having identified the forms of all possible critical points, what remains is to

\begin{enumerate}
\item describe the critical submanifolds and compute their Poincare polynomials, and
\item compute the Morse indices for each of the critical submanifolds
\end{enumerate}

\begin{remark}This is the place where the cases $n\ge 3$ become much more difficult than the cases $n=2$ and $n=3$. In particular, the main obstacle to progress is the description of the critical submanifolds and the computation of their Poincare polynomials.
\end{remark}

\noi We describe briefly how Gothen completed the Morse theory program in the case $n=3$, giving only enough details to highlight some of the geometry.  For full details we refer the reader to \cite{gothen94}.  The main idea is that the critical submanifolds can be identified as moduli spaces in their own right and for which the Poincare polynomials can be computed. 

Consider the Hodge bundles of type $(1,2)$.  The component $\varphi$ in the Higgs field is a holomorphic map 

\begin{equation}
\varphi:L\rightarrow V K
\end{equation} 

\noi or, equivalently, an element of the space of holomorphic sections $H^0(VL^{-1}K)$.  Denoting $\deg{L}$ by $l$, we see that these Hodge bundles are determined by pairs $(\mathcal{V},\varphi)$ where $\mathcal{V}$ is a rank three bundle of degree $d-l-(2g-2)$ and $\varphi\in H^0(\mathcal{V})$.  Just as for Higgs bundles, moduli spaces for pairs of this sort may be constructed if one introduces a suitable notion of stability. The appropriate definition depends on a real parameter (see \cite{bradlow91}) and the resulting moduli spaces are non-empty provided the parameter lies in a bounded range determined by the rank and degree of the bundle.   For rank 2 bundles the stability of a pair $(\mathcal{V},\varphi)$ can be formulated as a condition on line subbundles $\mathcal{L}\subset\mathcal{V}$, namely

\begin{equation}
\deg(\mathcal{L})< \begin{cases}\frac{v}{2}-\sigma\quad if\ \Phi\in H^0(\mathcal{L})\\
 \frac{v}{2}+\sigma\quad otherwise\end{cases}
\end{equation} 

\noi where $v=\deg(\mathcal{V})$ and $\sigma$ is a real parameter. If the inequalities are allowed to be equalities then the pair is $\sigma$-semistable.  The allowed range for $\sigma$ is

\begin{equation}\label{sigmarange}
0\le\sigma\le\frac{v}{2} \ .
\end{equation}

\noi This range is partitioned at the values of $\sigma$ for which  $\frac{v}{2}\pm\sigma$ is integral\footnotemark\footnotetext{Thus at integer or half-integer values, depending on whether $\deg{V}$ is even or odd}.  Except for values of $\sigma$ at these partition points, the moduli spaces of $\sigma$-stable pairs, say $N_{\sigma}(2,v)$, are smooth projective varieties. In fact the moduli spaces depend only on the subinterval containing $\sigma$, so there is a discrete family of distinct moduli spaces labelled by the subintervals $(\frac{v}{2}-i-1,\frac{v}{2}-i)$ (for $i=0,1,\dots i_{max}$).  

It is not hard to show (see \cite{gothen94}) that the Hodge bundle $(L\oplus V,\varphi)$ is stable if and only if the corresponding pair $(\mathcal{V},\varphi)$ is $\sigma$-stable for a specific value of $\sigma$, namely

\begin{equation}\label{sigma}
\sigma=-\frac{d}{6}+\frac{l}{2}
\end{equation}

\noi We thus get a precise relationship between the critical points of type $(1,2)$ and a moduli space of $\sigma$-stable pairs  induced by the map
\begin{equation}
(L\oplus V,\varphi)\mapsto (VL^{-1}K,\varphi)
\end{equation}

\noi It remains to determine the Poincare polynomial for the moduli space $N_{\sigma}(2,v)$ where $v=d-l+2g-2$ and $\sigma$ is given by \eqref{sigma}.  This can be done by adapting the work of Thaddeus (\cite{Thaddeus94}).  The moduli space corresponding to the largest subinterval in the range for $\sigma$ is the easiest to describe and the others can be related to it by  `flips', i.e. explicit transformations whose effect on the Poincare polynomial can be computed. 

The Hodge bundles of type $(2,1)$ can be analysed in essentially the same way, while those of type $(1,1,1)$ can be described as finite coverings of products of Jacobian varieties.   The Morse indices at each of the critical points are calculated using their relation to the weight spaces for the $S^1$ action (see Proposition \ref{wtspaces}). 

Putting everything together, Gothen computes the full Poincare polynomial for $\mathcal{M}_d$ for any $d$ coprime to 3 (see \cite{gothen94}, Theorem 1.2).

\begin{remark}For $\GL(n,\C)$-Higgs bundles with $n>3$, if the degree of the bundle is coprime to $n$ then the moduli space $\mathcal{M}_d(\GL(n,\C))$ is smooth,  the Morse indices for the critical submanifolds of the Hitchin function can be computed, and the submanifolds can be described as moduli spaces in their own right.  The correspond to moduli spaces of quiver bundles where the quivers describe the structure of the possible Hodge bundles that occur at fixed points of the $S^1$-action.  Unfortunately, the Poincare polynomials of these spaces are not known.  
\end{remark}

\subsection{Good enough for $\pi_0$} \label{subs:goodenough}

The success of the full Morse theory program for moduli spaces of $\GL(3,\C)$-Higgs bundles depends on the smoothness of the moduli spaces and the fact that the critical submanifolds can be described and have computable Poincare polynomials. In this section we discuss examples of $G$-Higgs bundles that fall short of these requirements but in which we can at least compute the first topological invariant, namely the number of connected components.  In these cases $G$ is a non-compact real reductive group, so the $G$-Higgs bundles are as described in Section \ref{GHiggs}.  Many specific examples have been examined (see \cite{arroyo},\cite{bradlow-garciaprada-gothen2003},\cite{bradlow-garciaprada-gothen2004},\cite{garcia-gothen-mundet:2008},\cite{garcia-prada-mundet:2004},\cite{garcia-prada-oliveira},\cite{gothen01},\cite{Hitchin87, Hitchin92})

We discuss just one case, namely $G=\Sp(2n,\R)$.  We pick this illustrative example because $\Sp(2n,\R)$ is both a split real form and also has the property that the associated homegeneous space $G/H$ is a Hermitian symmetric space.  Moreover, since $\Sp(2n,\R)$ is semisimple, we can invoke the identification of \eqref{M=R} to identify the Higgs bundle moduli spaces with moduli spaces of surface group representations in $\Sp(2n,\R)$.

For $G=\Sp(2n,\R)$ we have seen in Example \ref{Ex;Sp2nR} that an $\Sp(2n,\R)$-Higgs bundle on $X$ consists of a triple $(V,\beta,\gamma)$ where $V$ is a rank $n$ holomorphic bundle on $X$ and the components of the Higgs field are symmetric maps
\begin{align}
\beta:V^*\rightarrow V\otimes K\\
\gamma:V\rightarrow V^*\otimes K
\end{align}
The topological invariant is $d=\deg(V)$.  The definitions of stability and polystability for $\Sp(2n,\R)$, given in full in \cite{garcia-gothen-mundet:2008}, amount to an inequality on the degrees of subbundles in 2-step filtrations of $V$ that are compatible with the Higgs field.  For polystable $\Sp(2n,\R)$-Higgs bundles the invariant $d$ has to satisfy the bound:

\begin{equation}\label{MW}
0\le |d|\le n(g-1)
\end{equation} 

\noi For each value of $d=\deg(V)$ in the range \eqref{MW}, we get a moduli space of polystable $\Sp(2n,\R)$-Higgs bundles, denoted by $\mathcal{M}_d(\Sp(2n,\R))$. These spaces are complex algebraic varieties of dimension $(g-1)(2n^2+n)$ but in general they have singularities (at decomposable $\Sp(2n,\R)$-Higgs bundles). This is the first obstacle to full implementation of the Morse theory program.  We can nevertheless describe the $\Sp(2n,\R)$-Hodge bundles, i.e. the critical points of the $S^1$-action on $\mathcal{M}_d(\Sp(2n,\R))$ given by \eqref{S1action}, and identify the minima.   

The $\Sp(2n,\R)$-Higgs bundle equation becomes

\begin{equation}\label{Speqtn}
F_A +\beta\beta^*-\gamma^*\gamma=0
\end{equation}

\noi where $A$ is a connection on $V$ and the adjoints on $\beta,\gamma$ are taken with respect to a fixed metric on $V$.  The Hitchin function is given in this situation by

\begin{equation}\label{SpHitchin1}
f(V,\beta,\gamma)=||\beta||^2+||\gamma||^2
\end{equation}

\noi where the $L^2$ norms are computed with respect to the fixed metric on $V$.  Combining \eqref{Speqtn} and \eqref{SpHitchin1} yields

\begin{equation}\label{SpHitchin2}
f(V,\beta,\gamma)=-d+2||\beta||^2 = d+2||\gamma||^2
\end{equation}

\noi Define the loci $\mathcal{N}_d\subset\mathcal{M}_d(\Sp(2n,\R))$ by

\begin{equation}
\mathcal{N}_d = \{(V,\beta,\gamma)\in\mathcal{M}_d(\Sp(2n,\R))\ | \beta=0\ if\ d\le 0, or\ \gamma=0\ if\ d\ge 0\}
\end{equation}

\noi It follows from \eqref{SpHitchin2} that $\mathcal{N}_d$ is a subset of $f_{min}$, the locus of global minima of $f$.  It does not follow - and in fact is not always true - that all global minima lie in $\mathcal{N}_d$.  

The Hodge bundles are as in Definition \ref{defn:SpHodge}, with further constraints  imposed by stability on the ranks and degrees (see Corollary 6.6 in \cite{garcia-gothen-mundet:2008}).  Using the criterion described in Section \ref{subs:min}, we thus get the following characterization of the minima. 

\begin{proposition}\label{prop:smoothmin}  Let $(V,\beta,\gamma)$ be a fixed point of the $S^1$ action which lies in the smooth locus of $\mathcal{M}_d(\Sp(2n,\R))$.
\begin{enumerate}
\item If $|d|< n(g-1)$ then $(V,\beta,\gamma)\in f_{min}$ if and only if $(V,\beta,\gamma)\in \mathcal{N}_d$.
\item If $d=-n(g-1)$ then $(V,\beta,\gamma)\in f_{min}$ if and only if $(V,\beta,\gamma)\in \mathcal{N}_d$ or it is a Hodge bundle with
\begin{equation}
V=\begin{cases} \bigoplus_{j=-q}^qL^{-1}K^{-2j}\ if\ n=2q+1\\
\bigoplus_{j=-q}^{q+1}LK^{-2j}\quad if\ n=2q+2\ne 2
\end{cases}
\end{equation}

\noi where $L$ is a line bundle such that $L^2=K$ and, with respect to this decomposition for $V$ and the corresponding decomposition of $V^*$, 

\begin{equation}
\beta=\begin{pmatrix}0&\dots&0&1\\
\vdots&\ddots&\ddots&0\\
0&1&\ddots&\vdots\\
1&0&\dots&0 \end{pmatrix}\quad and\quad \gamma=\begin{pmatrix}0&\dots&0&0\\
\vdots&\ddots&\ddots&1\\
0&0&\ddots&\vdots\\
0&1&\dots&0 \end{pmatrix}
\end{equation}
\item If $d=n(g-1)$ the result is the same as for $d=-n(g-1)$ but with $(V,\beta,\gamma)$ replaced by $(V^*,\gamma^t,\beta^t)$.
\end{enumerate} 
\end{proposition}

\begin{remark}\label{remark:n=2} In the case $n=2$ and $|d|=2g-2$, i.e. for maximal $\Sp(4,\R)$-Higgs bundles, there are additional possibilities for the Hodge bundles, as described below in Section \ref{subsub:4}. In the case $n=1$ we have $\Sp(2,\R)\simeq\SL(2,\R)$, a case analyzed by Hitchin in\cite{Hitchin87}.  
\end{remark}

This result is not sufficient to identify all minima because fixed points of the $S^1$-action can occur at singular points.  One has to examine `by hand' whether such critical points are minima of the Hitchin functional. In the case of $\Sp(2n,\R)$-Higgs bundles, generalizing Hitchin's results in \cite{Hitchin87} for $\SL(2,\R)$, one can explicitly construct paths in $\mathcal{M}_d(\Sp(2n,\R))$ through such Hodge bundles and along which the Hitchin function is decreasing. It follows that the minima all lie in the smooth locus and hence are all of the form given in Proposition \ref{prop:smoothmin}.

Having characterized the minima of the Hitchin function on the spaces $\mathcal{M}_d(\Sp(2n,\R))$, we can now - in principle - use Theorem \ref{th:proper} to investigate the number of connected components. 

In the case $|d|<n(g-1)$, Proposition \ref{prop:smoothmin} says that we need to examine only the space $\mathcal{N}_d$.  The objects in this space are so-called quadratic pairs $(V,q)$ where $q$ is a (possibly degenerate) quadratic form on $V$ or $V^*$ with values in the line bundle $K$. A good notion of stability can be formulated for such objects and moduli spaces can be constructed (see \cite{garcia-gothen-mundet:2008}). At present the only case for which the connectedness of these spaces is known is the case of rank two quadratic pairs, in which case they are connected (see \cite{garcia-gothen-mundet:2008} and \cite{gothen-oliveira}) and we get

\begin{proposition} The moduli spaces $\mathcal{M}_d(\Sp(4,\R))$ are connected for all $|d|<2g-2$
\end{proposition}

\begin{remark} It is reasonable to conjecture that for all $n$ and all $0\le |d| <n(g-1)$ the moduli spaces $\mathcal{M}_d(\Sp(2n,\R))$ are connected.
\end{remark}

\noi The situation is better for extremal values of $d$, i.e. in the cases where $d=0$ or $|d| =n(g-1)$.  In the case $d=0$, the critical points all have $\beta=\gamma=0$ and we can identify $f_{min}=\mathcal{N}_0$ with $M(n,0)$, the moduli space of polystable rank $n$ vector bundles.  Since $M(n,0)$ is known to be connected, so therefore is $\mathcal{M}_0(\Sp(2n,\R))$.

In the case $|d| =n(g-1)$ we see two new phenomena:

\begin{enumerate}
\item Multiple components.  The two cases in parts (2) and (3) of Proposition \ref{prop:smoothmin} already show that the moduli spaces $\mathcal{M}_{\pm n(g-1)}(\Sp(2n,\R)$ have multiple components. In fact, there are at least as many components of the second type as there are choices for the line bundle $L$, i.e. for a square root of $K$. We discuss this in Section \ref{subsub:cayley}
\item Teichmuller components. In the case of $\Sp(2,\R)=\SL(2,\R)$, it is well known (see \cite{Goldman88} and also \cite{Hitchin87}) that some\footnotemark\footnotetext{For a surface of genus $g$ the number is $2^{2g}$, to be precise.} of the components of $\mathcal{M}_{g-1}(\Sp(2,\R)$ can be identified with Teichmuller space. For $n\ge 2$ a subset of the extra maximal components share key properties similar to Teichmuller space and are referred to as `higher Teichmuller' or Hitchin components.  We discuss this in Section \ref{subsub:teich}
\end{enumerate}

\subsubsection{Cayley transform and new invariants. } \label{subsub:cayley}

A distinguishing feature of polystable $\Sp(2n,\R)$-Higgs bundles for which $|d|=n(g-1)$ is that one of the two parts of the Higgs field has to have maximal rank as a bundle endomorphism (see Proposition 3.22 of \cite{garcia-gothen-mundet:2008}).  For example, if $(V,\beta,\gamma)$ is polystable with $\deg(V)=n(g-1)$ then $\gamma:V\rightarrow V^*\otimes K$ is an ismorphism.\footnotemark\footnotetext{If $d=-n(g-1)$ then $\beta:V^*\rightarrow V\otimes K$ must be an isomorphism.}  This is manifestly true for the Hodge bundles described in parts (2) and (3) of Proposition \ref{prop:smoothmin} in which both $\beta$ and $\gamma$ are non-trivial but applies equally to other Hodge bundles or the Higgs bundles which are not critical points.  Notice that for all such Higgs bundles, say with positive maximal $d$, the map $\gamma$ is a symmetric non-degenerate bundle map. After twisting by a line bundle $K^{-1/2}$ (i.e. a square root of  $K^{-1}$), we thus get a bundle $W=VK^{-1/2}$ with a symmetric, non-degenerate map 

\begin{equation}
q=\gamma\otimes 1_{K^{-1/2}}:W\rightarrow W^*
\end{equation}

\noi The pair $(W,q)$ defines an orthogonal bundle, i.e. a bundle with structure group $\mathrm{O}(n,\C)$. 
As a result of this emergence of a new structure group new characteristic classes, namely the Stiefel-Whitney classes for the orthogonal bundle, separate $\mathcal{M}_{n(g-1)}(\Sp(2n,\R)$ into components with fixed values of the new classes. 
For orthogonal bundles over a Riemann surface $X$ there are two Stiefel-Whitney classes with values in $H^1(X,\Z_2)$ and $H^2(X,\Z_2)$ respectively, and thus a total of $2.2^{2g}$ possibilities.  Notice, furthermore that in each of these components the Higgs bundles in the minimal submanifold $\mathcal{N}_{n(g-1)}$ have $\beta=0$ and are thus fully determined by the data defining the orthogonal bundles.  This means that $\mathcal{N}_{n(g-1)}$ can be identified with the moduli space of polystable principal $\mathrm{O}(2n,\C)$-bundles and hence, by the results of \cite{ramanathan:1996}, is connected. In addition, there are the components on which the minima are single points represented by the Hodge bundles of the type described in parts (2) and (3) of Proposition \ref{prop:smoothmin}.  Since there are $2^{2g}$ choices for the line bundle $L$ (i.e. for the square root of $K$), we get finally:

\begin{theorem}[\cite{garcia-gothen-mundet:2008}, Theorem 7.14] On a genus $g$ closed Riemann surface,  $\mathcal{M}_{\pm n(g-1)}$ has $3.2^{2g}$ connected components if $n\ge 3$.
\end{theorem}

Much of the above discussion reflects general features seen in those cases where the group $G$ has the property that the associated symmetric space $G/H$ is of Hermitian type. \footnotemark\footnotetext{see\cite{bradlow-garciaprada-gothen2006}  for further  examples} In that context the invariant $d$ is known as the {\bf Toledo invariant} and the bound \eqref{MW}  is called the {\bf Milnor-Wood bound}.  If $G$ is of Hermitian type then the Higgs fields for $G$-Higgs bundles with maximal Toledo invariant specialize in a way that results in the appearance of a new structure group called the {\bf Cayley partner} to $G$.  This, in turn, identifies new invariants which distinguish multiple components of $\mathcal{M}_{max}(G)$, the moduli space of polystable $G$-Higgs bundles with maximal Toledo invariant. 

Under the identification \eqref{M=R}, the $G$-Higgs bundles with maximal Toledo invariant correspond to special surface group representations in $G$.  Such maximal representations have been intensively studied by a variety of methods and are known to have special dynamical and geometric properties (see \cite{BILW05, BIW10,BIW11} for surveys). While many features are seen most easily from the point of view of surface group representations, to-date the Morse theory methods described above provide the most effective tools for counting components of the moduli spaces.

\subsubsection{Hitchin components} \label{subsub:teich} The Hodge bundles of parts (2) and (3) in Proposition \ref{prop:smoothmin} illustrate a phenomenon that can be traced back to the fact that the group $\Sp(2n,\R)$ is a {\it split real form} of a complex reductive group.  In the case of the split real form $\PSL(2,\R)$, a hyperbolic structure on $X$ determines a (Fuchsian) representation of $\pi_1(X)$ in $\PSL(2,\R)$ and this leads to the identification of a component of $\Rep(\pi_1(X),\PSL(2,\R))$ with the Teichmuller space for the surfaces of genus $g$ (see \cite{Goldman88}). In \cite{Hitchin87} Hitchin showed how to identify the components in $\mathcal{M}_{g-1}(\PSL(2,\R))$ which correspond to these representations. The minima of the Hitchin function on these components can be identified as  the $n=1$ case of the  Hodge bundles of part (2) in Proposition \ref{prop:smoothmin}.  In the case of other split real forms $G$ (such as $\SL(n,\R), \Sp(2n,\R), \SO(n,n+1),\SO(n,n)$), the Lie algebra contains principal three dimensional subgroups which define homomorphisms of $\PSL(2,\R)$ into $G$.  Composing this embedding with the Fuchsian representations leads to so-called higher Teichmuller components in $\Rep(\pi_1(X),G)$ or correspondingly in the the moduli space of $G$-Higgs bundles.  In the case of $\Sp(2n,\R)$, these are precisely the components for which the minima of the Hitchin function are as in part (2) in Proposition \ref{prop:smoothmin}. 

\begin{remark} The Hitchin components can also be identified by means of another map defined by Hitchin. Unlike the function define in Section \ref{sect:hitchinfn}, this is a map 

\begin{equation}
h:\mathcal{M}(G)\rightarrow \bigoplus_i H^0(K^{p_i})
\end{equation}
from the moduli space $\mathcal{M}(G)$ to a linear space given by $\bigoplus_i H^0(K^{p_i})$ where the powers $p_i$ are the degrees of Ad-invariant homogeneous polynomials which generate the ring of all such polynomials on the summand $\mclie\subset\lieg$. Here, as usual, $K$ is the canonical bundle on $X$ and $H^0$ denotes the space of holomorphic sections. In \cite{Hitchin92} Hitchin showed that when $G$ is a split real form, this map admits sections. The image of these sections are precisely the Hitchin components. 
\end{remark}

\subsubsection{The case of $\Sp(4,\R)$} \label{subsub:4} We end with a brief account of the special case of $G=\Sp(4,\R)$ (see \cite{bradlow-garciaprada-gothen2011}, \cite{garcia-prada-mundet:2004},  \cite{gothen01} for more details) . In this case, the Higgs bundles are defined by triples $(V,\beta,\gamma)$ where $V$ is a rank 2 bundle of degree $d$ with $0\le |d|\le 2g-2$.  For the Higgs bundles with $|d|=2g-2$, say $d=(2g-2)$ for definiteness,  the bundle in the Cayley partner is an orthogonal bundle of rank two, i.e. it has structure group $\OO(2,\C)$ and thus topological type determined by Stiefel-Whitney classes $w_1\in H^1(X,\Z_2)$ and $w_2\in H^2(X,\Z_2)$.  By a classification result of Mumford (\cite{mumford}) for such bundles:

\begin{proposition}\label{prop:O(2)bundles} The Higgs bundle $(V,\beta,\gamma)$ in $\mathcal{M}_{2g-2}(\Sp(4,\R)$ can  be taken to be one of the following types:

\begin{enumerate}
\item $V=L\oplus L^{-1}K$, where $L$ is a line bundle on $X$, and with respect to this decomposition,
\begin{equation}\label{bandg}
\gamma=\left(
\begin{smallmatrix}
0 & 1 \\
1 & 0
\end{smallmatrix}\right)\in H^0(S^2V^*\otimes K)
\quad\text{and}\quad
\beta=\left(
\begin{smallmatrix}
\beta_1 & \beta_3 \\
\beta_3 & \beta_2
\end{smallmatrix}\right)\in H^0(S^2V\otimes K).
\end{equation}
In this case the first Steifel-Whitney classes of the Cayley partner vanishes.   If  $(V,\beta,\gamma)$ is a polystable Higgs bundle then $\deg(L)$ must lie in the range $g-1\le \deg(L)\le 3g-3$ and $\beta_2\ne 0$ if $l>g-1$.
 
\item $V=\pi_*(\tilde{L}\otimes \iota^*\tilde{L}^{-1})K^{1/2}$ where
$\pi:\tilde{X}\longrightarrow X$ is a connected double cover, $\tilde{L}$ is a
line bundle on $\tilde{X}$, and $\iota:\tilde{X}\longrightarrow
\tilde{X}$ is the covering involution.  In this case the first Steifel-Whitney classes of the Cayley partner, $w_1 \in H^1(X;\Z/2)$, is the non-zero element defining the double cover.

\item $(V,\beta,\gamma)=(V_1,\beta_1,\gamma_1)\oplus (V_2,\beta_2,\gamma_2)$ where $(V_i,\beta_i,\gamma_i)$ are maximal $\Sp(2,\R)$-Higgs bundles.  In this case  the first Steifel-Whitney classes of the Cayley partner is the sum of the first Stiefel-Whitney classes for the Cayley partners to $(V_i,\beta_i,\gamma_i)$.
\end{enumerate}
\end{proposition}

\noi As in Section \ref{subsub:cayley}, the Stiefel-Whitney classes of the Cayley partner label separate components of $\mathcal{M}_{2g-2}(\Sp(4,\R)$. Moreover, in case (1) of Proposition \ref{prop:O(2)bundles}, i.e. when $w_1=0$, the degree of the line bundle $L$ constitutes an additional invariant\footnotemark\footnotetext{This can also be realized as an integral lift of the second Steifel-Whitney class $w_2\in H^2(X,\Z_2)$.}  We thus get a decomposition of $\mathcal{M}_{2g-2}(\Sp(4,\R)$ into:

\begin{equation}\label{comps1}
\mathcal{M}_{2g-2}(\Sp(4,\R)=(\bigcup_{w_1\ne 0,w_2}\mathcal{M}_{w_1,w_2})\cup
(\bigcup_{0 \leq l \le 2g-2}\mathcal{M}^0_l)
\end{equation}

\noi where $\mathcal{M}_{w_1,w_2}$ denotes the component in which the Higgs bundles are represented by $(V,\beta,\gamma)$ of type (2) with Stiefel-Whitney classes $w_1,w_2$, and $\mathcal{M}^0_l$ denotes the component in which the Higgs bundles are represented by $(V,\beta,\gamma)$ of type (1) in Proposition \ref{prop:O(2)bundles} with $\deg(L)=l$. Notice that when $l=3g-3$, the component $\beta_2$ in \eqref{bandg} must be a non-zero section of $L^{-2}K^3$, which means that 
\begin{equation}\label{eqn:Lchoices}
L^2=K^3
\end{equation}

\noi and $\beta_2$ can be taken to be the identity section of the trivial line bundle.  The component $\mathcal{M}^0_{3g-3}$ thus splits further into components determined by the $2^{2g}$ choices of $L$ satisfying \eqref{eqn:Lchoices}.  Denoting these by $K^{3/2}$, we get

\begin{equation}\label{comps2}
\mathcal{M}^0_{3g-3}=\bigcup_{K^{3/2}}\mathcal{M}_{K^{3/2}}
\end{equation}

\noi The components identified in \eqref{comps1} and \eqref{comps2} are not a priori connected, but we can use the Hitchin function (and Theorem \ref{th:proper}) to investigate their connectedness. The minima of the Hitchin function on each of these components can be identified as follows:

\begin{enumerate}
\item For $g-1<l\le 3g-3$, $(V,\beta,\gamma)$ represents a minimum of the Hitchin function on $\mathcal{M}^0_l$ if and only if it isomorphic to a Higgs bundle of the form in (1) of Proposition \ref{prop:O(2)bundles} with $\beta_1=\beta_2=0$,
\item On all other components, i.e. $\mathcal{M}^0_{g-1}$ or $\mathcal{M}_{w_1,w_2}$, $(V,\beta,\gamma)$ represents a minimum of the Hitchin function if and only if $\beta=0$.
\end{enumerate}

Notice that the minima on $\mathcal{M}_{K^{1/2}}$ are precisely the Hodge bundles identified in Proposition \ref{prop:smoothmin}, while the minima on the components $\mathcal{M}^0_{\nu}$ are additional Hodge bundles with $\beta\ne 0$.  These are the exceptional Hodge bundles referred to in Remark \ref{remark:n=2}.   The components $\mathcal{M}_{K^{1/2}}$ are the Hitchin components for $\Sp(4,\R)$ (see Section \ref{subsub:teich}) while the other components $\mathcal{M}^0_l$ have no analogs in $\mathcal{M}(\Sp(2n,\R)$ for $n\ne 2$.

The connectedness of the minimal submanifolds, and hence of the components, is proved in \cite{gothen01}.  It is interesting to note that in many cases, the minimal submanifolds can be shown to be connected by explicitly describing the entire loci: see \cite{bradlow-garciaprada-gothen2011} for descriptions of $\mathcal{M}^0_l$ for $g-1< l<2g-2$ and the Hitchin components $\mathcal{M}_{K^{3/2}}$.  We thus arrive at a final count of connected components:

\begin{theorem} The moduli space $\mathcal{M}_{2g-2}(\Sp(4,\R)$ has

\begin{displaymath}
2(2^{2g}-1)+ (2g-2) + 2^{2g} = 3\cdot 2^{2g} +2g - 4\ .
\end{displaymath}

connected components.
\end{theorem}


\end{document}